%% file: article.tex
\documentclass[onefignum,onetabnum]{siamart251216}

\ifpdf
\hypersetup{
  pdftitle={High-order entropy-stable isothermal wall boundary condition for the compressible Navier--Stokes equations},
  pdfauthor={Mohammed Sayyari and Nail K. Yamaleev}
}
\fi

\usepackage[utf8]{inputenc}
\usepackage{amsmath} 
\usepackage{amsfonts}
\usepackage{amssymb}
\usepackage{subcaption}
\usepackage{array}
\usepackage{etoolbox}
\usepackage{tikz}
\usepackage{hyperref}
\usepackage{graphicx}
\graphicspath{{figures/}}

\usetikzlibrary{decorations.pathreplacing,calligraphy}

\newtheorem{note}{Note}
\newtheorem{remark}{Remark}

\author{Mohammed Sayyari and Nail K. Yamaleev}
\title{High-order entropy-stable isothermal wall boundary condition for the compressible Navier--Stokes equations}

\input{positivity3D_def}

\begin{document}


\maketitle

\begin{abstract}
  Entropy-stable boundary conditions are critical for ensuring that the corresponding numerical scheme satisfies the discrete entropy inequality, thus mimicking the second law of thermodynamics for the compressible Navier-Stokes equations discretely. Provably entropy-stable adiabatic wall boundary conditions were introduced in Parsani 2014~\cite{parsani2014entropy}. These discrete wall boundary conditions were further generalized in Dalcin 2019~\cite{dalcin2019conservative} to a moving adiabatic solid wall or a wall with a prescribed heat flux for the compressible Navier--Stokes equations discretized by using summation-by-parts (SBP) and simultaneous-approximation-term (SAT) operators. Currently, no entropy-stable formulation using SBP-SAT operators for isothermal wall boundary conditions is available in the literature. This paper presents new isothermal no-slip wall boundary conditions that: 1) enforce the $T|_{y=0}=T^{wall}$ and $\vec{v}|_{y=0}=\vec{v}^{wall}$ conditions with the design order of accuracy while maintaining stability, 2) provide the correct sign of entropy production at the wall, and 3) mimic the entropy balance of the Navier-Stokes equations at the discrete level. The structure-preserving and design-order properties of the proposed methodology are demonstrated and verified on standard benchmark problems for compressible flows.
\end{abstract}




\section{Introduction}
The isothermal boundary condition is common in fluid dynamics and is used in theoretical, numerical, and experimental studies of convective heat transfer around solid objects~\cite{schlichting2016boundary,lighthill1950contributions} and cavities~\cite{ostrach1972natural,paolucci1994differentially}. While most of the literature focuses on incompressible flows, compressible flows are also of interest under isothermal conditions. This is especially important for viscous and high-speed flows around objects where temperature gradients are significant~\cite{anderson2006hypersonic,spina1994physics}. These gradients pose a significant challenge for numerical schemes, where sharp jumps may increase stiffness due to additional physical and artificial viscosity and positivity limitations, especially in high-order schemes due to the Gibbs phenomenon~\cite{persson2006sub,pirozzoli2011numerical,wang2013high}. In addition, high-order entropy-stable schemes require more stringent conditions to ensure that only physical entropy production occurs and that both the boundary condition and the scheme as a whole are design-order accurate and provably entropy-stable~\cite{parsani2014entropy,dalcin2019conservative,chan2022entropy}.

Although the focus is on the isothermal boundary condition, it is not alone in this challenge; the adiabatic and no-slip wall boundary conditions have undergone consistent development to ensure entropy stability and design-order accuracy. In~\cite{parsani2014entropy}, novel entropy-stable no-slip boundary conditions were developed and verified. These boundary conditions were weakly enforced using Simultaneous Approximation Term (SAT) penalties, with a manufactured viscous-wall state. In~\cite{dalcin2019conservative}, a novel local discontinuous Galerkin (LDG) formulation was developed to compute the entropy gradient with a penalty term for the entropy jump across interfaces and boundaries. With that, the no-slip boundary condition was extended to (inspirational) moving wall boundary conditions and heat-entropy flux imposition. In addition to the viscous wall state, the entropy wall state is also constructed using the equivalent entropy variables. Both \cite{parsani2014entropy,dalcin2019conservative} note that the imposition of the isothermal boundary condition via the viscous state is not entropy-stable and would result in positive entropy production at the wall. Thus, no further analysis or development of the isothermal wall boundary condition was done in these works. Furthermore, in~\cite{chan2022entropy}, the wall boundary conditions in~\cite{dalcin2019conservative} inspired their development for modal discontinuous Galerkin entropy-stable high-order schemes. There, in addition to the no-slip and adiabatic wall boundary conditions, the authors also implemented an isothermal boundary condition with the focus on mimicking the continuous entropy balance using an isothermal consistent entropy variable wall state, $W_i^\text{wall}=\frac{2 u_i^\text{wall}}{T^\text{wall}} - W_i$, $i=2,3,4$, and $W_5^\text{wall}=-\frac{2}{T^\text{wall}}-W_5$. The authors in~\cite{chan2022entropy} reported that, while the proposed discrete isothermal no-slip wall boundary condition mimics the continuous entropy inequality, it is not provably entropy-stable and can invariably lead to positive entropy production. In addition, the accuracy and order properties of this boundary condition were not verified.

Moreover, the isothermal boundary condition in~\cite{chan2022entropy} cannot be adapted to SBP-SAT-based schemes due to a stringent temperature positivity requirement imposed on the equivalent primitive variables, $T^\text{wall}< 2 T$. This boundary condition also does not resolve the entropy-production issue presented in~\cite{parsani2014entropy,dalcin2019conservative} for SBP-SAT-based schemes. Therefore, this paper will not build on the work in~\cite{chan2022entropy}. Instead, herein, a new isothermal wall boundary condition is developed based on the work in~\cite{dalcin2019conservative} that: 1) enforces the $T|_{y=0}=T^{wall}$ and $\vec{v}|_{y=0}=\vec{v}^{wall}$ conditions with design-order accuracy while maintaining stability, 2) provides the correct sign of entropy production at the wall, and 3) mimics the entropy balance of the Navier-Stokes equations at the discrete level. The order properties and entropy boundedness of this novel isothermal boundary condition are verified using 2D internal and external compressible flow cases. The robustness and accuracy are verified for a $Ma=10$ viscous flow around a cylinder, and a $Kn=0.002$ thermal Couette flow, respectively.

The paper is organized as follows. The compressible Navier--Stokes equations and the entropy inequality are presented in Section~\ref{sec:compressible-Navier-Stokes}. The entropy bound for the proposed isothermal boundary condition is analyzed in~\ref{sec:entropy-bounded-isothermal}. The entropy-stable high-order scheme is presented in Section~\ref{sec:entropy-stable-high-order-scheme}, where the positivity-preserving hybrid scheme and the entropy-stable wall boundary conditions are described. The focus of this paper, the isothermal-consistent wall heat flux, is derived in Section~\ref{subsec:isothermal-wall-heat-flux}. Section~\ref{sec:numerical-results} presents numerical results, including convergence studies proving the design-order property for a channel and flow over a cylinder, subsonic flow around a square cylinder, and hypersonic flow over a cylinder, demonstrating the versatility of this boundary condition treatment, and the Couette flow case, a classic case in the literature. Finally, conclusions are drawn in Section~\ref{sec:conclusions}.


\section{The compressible Navier--Stokes equations and the entropy inequality}
\label{sec:compressible-Navier-Stokes}

The compressible Navier--Stokes equations in curvilinear coordinates $(\xi_1,\xi_2, \xi_3)$ are given by

\begin{align}
    \label{eq:compressible-Navier--Stokes}
    \frac{\partial (JU)}{\partial t}
    + \frac{\partial}{\partial \xi_d}
    \left[
        J \frac{\partial \xi_d}{\partial x_i}
    \right]
    {\bm F}^{(I)}_i
    = \frac{\partial}{\partial \xi_d}
    \left[
        J \frac{\partial \xi_d}{\partial x_i}
    \right]
    {\bm F}^{(V)}_i,
\end{align}
where $\U = \left[\rho,\rho v_1,\rho v_2,\rho v_3,\rho\Efnc\right]^\top$ is a vector of the conservative variables ($\rho$- density, $\rho v_i$- momentum, and $\rho E$- total energy), $J=\left|\frac{\partial(x_1, x_2, x_3)}{\partial(\xi_1, \xi_2, \xi_3)} \right|$ is the metric Jacobian, and $(x_1, x_2, x_3)$ are the Cartesian coordinates. In Equation \eqref{eq:compressible-Navier--Stokes} and hereinafter, the Einstein summation convention is used $\cdot_d\cdot_d = \sum_{d}\cdot_d\cdot_d$. The inviscid and viscous fluxes of the Navier--Stokes equations, ${\bm F}^{(I)}_i$ and ${\bm F}^{(V)}_i$, are given by
\begin{align}
    \label{eq:inviscid-viscous-fluxes}
    &{\bm F}^{(I)}_i
    = \begin{bmatrix}
        \rho v_i \\
        \rho v_1 v_i + P\delta_{1i} \\
        \rho v_2 v_i + P\delta_{2i} \\
        \rho v_3 v_i + P\delta_{3i} \\
        (\rho\Efnc + P)v_i
    \end{bmatrix},
    &&
    {\bm F}^{(V)}_i
    = \begin{bmatrix}
        0 \\
        \tau_{1i} \\
        \tau_{2i} \\
        \tau_{3i} \\
        \tau_{ji}v_j 
        + \kappa \frac{\partial T}{\partial x_i}
    \end{bmatrix},
\end{align}
where $\delta_{ij}$ is the Kronecker delta, $P=\rho R_g T$ is the pressure, $R_g$ is the gas constant, $T$ is the temperature, and $\kappa$ is the heat conductivity. The stress tensor is given by
\begin{equation}
    \nonumber
    \tau_{ji} = \mu \left( \frac{\partial v_j}{\partial x_i} + \frac{\partial v_i}{\partial x_j} - \frac{2}{3}\delta_{ji}\frac{\partial v_l}{\partial x_l}\right),
\end{equation}
where $\mu$ is the dynamic viscosity. 

The compressible Navier--Stokes equations are equipped with the Godunov entro-\\
py-pair $(\mathcal{S}, \mathcal{F}) = (-\rho s, -\rho s \bm{v})$ with entropy variables
\begin{equation}
    W = \frac{dS}{dU} = \left[c_v-s-\frac{v_iv_i}{T},\frac{v_1}{T},\frac{v_2}{T},\frac{v_3}{T},-\frac{1}{T}\right]^\top
\end{equation}
where$s$ is the thermodynamic entropy. These equations satisfy the following entropy inequality, assuming the corresponding boundary conditions are entropy-stable:
\begin{equation}\label{eq:EinEquation3}
  \int_{\hat\Omega} 
  \frac{\partial ( J\mathcal{S})}{\partial t}
  \mr{d}\hat\Omega
  \leq \int_{\Gamma} -\kappa \frac{1}{T} \frac{\partial T}{\partial n} \mr{d}\Gamma,
\end{equation} 
where the adiabatic wall boundary condition corresponds to $\frac{\partial T}{\partial n}=0$ and thus zero entropy production at the wall. However, in the presence of an isothermal wall boundary condition, the right-hand side of the entropy inequality is non-zero, corresponding to a physical entropy production (or dissipation) term at the wall. Thus, the isothermal wall boundary condition is entropy-stable if the heat-entropy flux produced at the wall does not exceed the physical heat-entropy flux. 

For the sake of clarity, the curvilinear description using the metric Jacobian $J$ is omitted in the remainder of this paper, and the Cartesian description is used instead. The derivations and proofs can be recast in curvilinear coordinates without loss of generality.


\section{Structure-preserving isothermal boundary condition}
\label{sec:entropy-bounded-isothermal}

The discrete pe-\\
nalty term is introduced as a correction term to the temperature gradient. The resulting isothermal temperature gradient is constructed as
\begin{equation}
    \left(\frac{\partial T}{\partial x_i}\right)^\text{wall} 
    = \frac{\partial T}{\partial x_i} - \alpha (T - T^\text{wall}),
\end{equation}
where the internal heat flux is the term on the right-hand side of the entropy balance. Thus, by defining $g(t)$ as the temperature jump at the wall, the penalty can be applied as a heat-entropy flux source term as follows.
\begin{equation}
    g(t) = -\kappa \alpha (T - T^\text{wall}),
\end{equation}
where $\alpha$ is a penalty parameter and $\kappa$ is derived from the coefficient of the Fourier term in the energy equation. The discrete properties of this penalty are further discussed in the semi-discretization sections. For purely analytical purposes, this term is assumed to be added in the continuous equations. The purpose is to study the effects of the penalty term on the transition to an isothermal wall from an initial condition $T\neq T^\text{wall}$ required for some cases considered in the numerical results section.

The entropy balance, given $g(t)$, can be written as
\begin{equation}
    \label{eq:entropy-balance-gt}
    \int_{\Omega} \frac{\partial S}{\partial t} \, d\Omega
    \leq \int_{\Gamma} \left( -\kappa \frac{1}{T} \frac{\partial T}{\partial n} + \kappa \alpha \frac{(T - T^\text{wall})}{T} \right) \, d\Gamma.
\end{equation}
Noting that for $T,T^\text{wall} > 0$,
\begin{equation}
    \label{eq:temperature-diff}
    \frac{(T - T^\text{wall})}{T} = 1 - \frac{T^\text{wall}}{T} < 1.
\end{equation}
The entropy production at the wall is
\begin{align*}
    \int_{\Omega} \frac{\partial S}{\partial t} \, d\Omega
    &\leq \int_{\Gamma} \left( -\kappa \frac{1}{T} \frac{\partial T}{\partial n} + \kappa \alpha \right) \, d\Gamma.
\end{align*}
Thus, if the heat entropy flux is bounded, the time derivative of the entropy is bounded. It has been shown in~\cite{svard2015weak,sayyari2022capabilities,upperman2021high} that such a bound results in an $L^2$ bound on the solution. However, the temperature gradient term remains with no a priori estimate.

A ballistic energy estimate inspired by the work in~\cite{chaudhuri2022navier} is used to derive a bound that can be controlled. Assuming no-slip boundary conditions $v \cdot n = 0$ at the wall, and imposing $g(t)$ as a penalty term on the boundary, $\Gamma$, where $T^\text{wall}$ is constant along the wall. In the following lemma, a bound is established on the time derivative of the ballistic energy.

\begin{lemma}[Ballistic energy rate bound]
    \label{lem:ballistic-rate-bound}
    Define the ballistic energy as $B = \rho E + T^\text{wall} \mathcal{S}$. Assuming no-slip (including the tangentially moving wall) and $g(t)$ penalty at a wall, then
    \begin{align}
    \label{eq:ballistic-energy-bound}
            \frac{d}{dt} \int_\Omega B \, d\Omega
            &\leq
            \int_{\Gamma}
            \frac{\kappa}{4\alpha T} \left(\frac{\partial T}{\partial n}\right)^2
            \, d\Gamma.
    \end{align}
\end{lemma}

\begin{proof}
    Given the assumed boundary conditions, the proof begins by establishing the energy balance wall contribution.
    \begin{equation}
        \label{eq:energy-balance}
        \int_{\Omega} \frac{\partial (\rho E)}{\partial t} \, d\Omega
        = \int_{\Gamma} \left( \kappa \frac{\partial T}{\partial n} - \kappa \alpha (T - T^\text{wall}) \right) \, d\Gamma.
    \end{equation}
    Next, the total energy contribution and the entropy rate bound are combined to form the ballistic energy bound. Substituting the relationships in Equation~\ref{eq:energy-balance} and Equation~\ref{eq:entropy-balance-gt} into the time derivative of the domain integral of the ballistic energy yields the following combined bound.
    \begin{align*}
        \frac{d}{dt} \int_\Omega B \, d\Omega
        &\leq \int_{\Gamma} 
        \left( 
            \kappa \frac{\partial T}{\partial n} 
            - \kappa \alpha (T - T^\text{wall}) 
            - \kappa \frac{T^\text{wall}}{T} \frac{\partial T}{\partial n} 
            + \kappa \alpha T^\text{wall} \frac{(T - T^\text{wall})}{T}
        \right) \, d\Gamma\\
        &\leq \int_{\Gamma} 
        \left( 
            \kappa \left(1 - \frac{T^\text{wall}}{T}\right) 
            \frac{\partial T}{\partial n} 
            - \kappa \alpha \left(1 - \frac{T^\text{wall}}{T}\right) (T - T^\text{wall}) 
        \right) \, d\Gamma.
    \end{align*}
    The combined terms can be further simplified using the equality term in Eq.~\ref{eq:temperature-diff}.
    \begin{align*}
        \frac{d}{dt} \int_\Omega B \, d\Omega
        &\leq \int_{\Gamma} 
        \left( 
            \kappa \left(\frac{T - T^\text{wall}}{T}\right) 
            \frac{\partial T}{\partial n} 
            - \kappa \alpha \frac{(T - T^\text{wall})^2}{T}
        \right) \, d\Gamma\\
        &= \int_{\Gamma} -\kappa\frac{1}{T}
        \left( 
            -\left(T - T^\text{wall}\right) 
            \frac{\partial T}{\partial n} 
            + \alpha (T - T^\text{wall})^2
        \right) \, d\Gamma.
    \end{align*}
    To achieve the desired term on the boundary, the following perfect square is used.
    \begin{equation}
        PS =
        \left(
            - \frac{1}{2\sqrt{\alpha}} \frac{\partial T}{\partial n} 
            + \sqrt{\alpha} (T - T^\text{wall})
        \right)^2
        = \frac{1}{4\alpha} \left(\frac{\partial T}{\partial n}\right)^2 
        - (T - T^\text{wall}) \frac{\partial T}{\partial n}
        + \alpha (T - T^\text{wall})^2.
    \end{equation}
    Since the perfect square enters the bound with a negative sign, its contribution is non-positive. This leavs the desired term in~\eqref{eq:ballistic-energy-bound}, i.e.,
    \begin{align*}
        -\kappa\frac{1}{T}
        \left( 
            -\left(T - T^\text{wall}\right) 
            \frac{\partial T}{\partial n} 
            + \alpha (T - T^\text{wall})^2
        \right)
        &= -\kappa\frac{1}{T} PS
        + \kappa \frac{1}{4\alpha T} \left(\frac{\partial T}{\partial n}\right)^2\\
        &\leq
        \frac{\kappa}{4\alpha T} \left(\frac{\partial T}{\partial n}\right)^2,
    \end{align*}
    completing the proof.
\end{proof}

The ballistic energy bound contains both the temperature gradient and the reciprocal of the correction parameter $\alpha$. While the temperature gradient does not have any a priori estimates, the presence of the reciprocal of $\alpha$ allows us to numerically measure the convergence of the term as a whole in the results section. For the purposes of demonstrating the importance of this bound, the term on the boundary is assumed to be bounded only in the remainder of this section.

One of the main properties of satisfying the entropy inequality is that it provides an $L^2$ bound on the solution~\cite{svard2015weak,upperman2021high,sayyari2022capabilities}. This result is established by taking the time integral of the time derivative of the entropy, $\frac{\partial \mathcal{S}}{\partial t}$, the same approach applies to the ballistic energy inequality. To build on previous work, a lemma is introduced to establish the main bound used to translate the bound on the entropy rate into a bound on the solution.
\begin{lemma}[Solution difference bound]
    \label{lem:entropy_bounds_solution_diffrence}
    Let $\mathcal{S}$ be a convex entropy and define the affine entropy
    \begin{equation*}
        \bar{\mathcal{S}}\left(U(\mathcal{T})\right)=\mathcal{S}(\mathcal{T})-\mathcal{S}(0)-\frac{d \mathcal{S}(0)}{dU}^\top\left(U(\mathcal{T})-U(0)\right),
    \end{equation*}
    $\overline{U}$ is a convex combination of $U(\mathcal{T})$ and $U(0)$, and $\lambda_{\min}$ is the minimum eigenvalue of the Hessian of the entropy function, then
    \begin{equation}
        \int_\Omega \|U(\mathcal{T}) - U(0)\|_{L^2} \, d\Omega
        \leq
        \frac{2 \int_\Omega \bar{\mathcal{S}}(\mathcal{T}) \, d\Omega}{\lambda_{\min}\left(\frac{d^2 \mathcal{S}(\overline{U})}{dU^2}\right)},
    \end{equation}
\end{lemma}

\begin{proof}
    See the proof of Lemma 1 in ~\cite{sayyari2022capabilities}.
\end{proof}

Lemmas~\ref{lem:ballistic-rate-bound} and~\ref{lem:entropy_bounds_solution_diffrence} are then used to establish that the bound on the ballistic energy rate results in an $L^2$ bound on the solution.

\begin{theorem}[$L^2$ ballistic energy bound for the isothermal wall boundary condition]
    \label{thm:ballistic-energy-bound}
    Assume the positivity of thermodynamic variables. Further, assume that the boundary term
    \begin{equation}
        \mathcal{C}_\text{ballistic}(\alpha)
        = \int_{\Gamma}
        \frac{\kappa}{4\alpha T} \left(\frac{\partial T}{\partial n}\right)^2
        + \kappa \frac{T^\text{wall}}{T_0}\frac{\partial T}{\partial n}
        \, d\Gamma < \infty,
    \end{equation}
    is bounded, then the solution of the compressible Navier--Stokes equations with the isothermal wall boundary condition is bounded in the $L^2$ norm by initial and boundary data
    \begin{align}
        \label{eq:l2-bound}
        \begin{split}
            \int_\Omega \|U(\mathcal{T})\|_{L^2} \, d\Omega
            &\leq
            \frac{2}{\lambda_{\min}\left(\frac{d^2 \mathcal{S}(\overline{U})}{dU^2}\right) T^\text{wall}}
            \left[
                \int_\Omega \rho E(0) \, d\Omega
                + \int_0^\mathcal{T} C_\text{ballistic}(\alpha) \, dt
            \right]\\
            &+ \int_\Omega \|U(0)\|_{L^2} \, d\Omega + \mathcal{T}\int_\Gamma \kappa\alpha\frac{\left(T^\text{wall}\right)^2}{T_0}\ d\Gamma.
        \end{split}
    \end{align}
\end{theorem}

\begin{proof}
    The proof goes in two directions; the first establishes that the ballistic energy bounds the affine entropy. Lemma~\ref{lem:ballistic-rate-bound} provides the bound for the Godunov entropy $\mathcal{S}$. The constant initial condition terms of the affine entropy cancel under the time derivative, $\frac{\partial \mathcal{S}(0)}{\partial t}=0$ and $\frac{d \mathcal{S}(0)}{dU}^\top\frac{\partial U(0)}{\partial t}=0$. The rate of the affine entropy results in the additional term
    \begin{equation}
        \int_\Omega - T^\text{wall} \frac{d \mathcal{S}(0)}{dU}^\top \frac{\partial U}{\partial t}\ d\Omega
        \leq
        \int_\Gamma \kappa \frac{T^\text{wall}}{T_0}\frac{\partial T}{\partial n}
        - \kappa \alpha \frac{T^\text{wall}}{T_0} (T - T^\text{wall})\ d\Gamma.
    \end{equation}
    The first term is the second term in $C_\text{ballistic}(\alpha)$, the second is bounded by a constant
    \begin{equation}
        - \kappa \alpha \frac{T^\text{wall}}{T_0} (T - T^\text{wall}) 
        = -\kappa\alpha\frac{T^\text{wall} T}{T_0}
        + \kappa\alpha\frac{\left(T^\text{wall}\right)^2}{T_0}
        \leq \kappa\alpha\frac{\left(T^\text{wall}\right)^2}{T_0}.
    \end{equation}

    The second direction is bounding $\bar{\mathcal{S}}(U)$ by the ballistic energy. Since the total energy is positive,
    \begin{equation}
        \bar{\mathcal{S}}(\mathcal{T}) \leq \frac{\rho E}{T^\text{wall}} + \bar{\mathcal{S}}(\mathcal{T}) = \frac{B}{T^\text{wall}}
    \end{equation}
    
    To achieve the bound on the solution, the left-hand side of the result in Lemma~\ref{lem:entropy_bounds_solution_diffrence} is addressed. Computing the norm of $U(\mathcal{T})=U(\mathcal{T})-U(0)+U(0)$ and by the triangle inequality 
    \begin{equation}
        \int_\Omega \|U(\mathcal{T})\|_{L^2} \, d\Omega
        \leq
        \int_\Omega \|U(\mathcal{T}) - U(0)\|_{L^2} \, d\Omega
        + \int_\Omega \|U(0)\|_{L^2} \, d\Omega.
    \end{equation}

    Putting all the pieces together achieves the bound~\eqref{eq:l2-bound}.
\end{proof}

The affine entropy in this section is for analysis purposes only and is not used in the semi-discrete scheme. It demonstrates that, in addition to the bound on the entropy rate, the bound on the ballistic energy bounds the $L^2$ norm of the solution, thus the important bound to uphold is Lemma~\ref{lem:ballistic-rate-bound}. The result of this section shows that not only does the isothermal wall boundary condition treatment not violate the entropy inequality, it also allows for deriving the $L^2$ bound on the solution.


\section{Entropy-stable high-order scheme}
\label{sec:entropy-stable-high-order-scheme}

The entropy-stable high-order scheme is constructed using SBP operators and SAT procedure to discretize the compressible Navier--Stokes equations in curvilinear coordinates. SBP operators allow for mimicking the integration-by-parts property at the discrete level, resulting in a provably entropy-stable scheme when the boundary conditions are also entropy-stable. The SBP-SAT framework is an analysis tool that allows for a systematic construction of high-order-accurate and provably entropy-stable schemes for the compressible Navier--Stokes equations~\cite{fernandez2014review}. The scheme is built on a discontinuous collocation of Legendre-Gauss-Lobatto (LGL) nodes in each element, and SAT terms are used to weakly enforce the internal interface and boundary conditions. This section will briefly review the SBP operators used in this work, with a focus on the wall boundary conditions derived in~\cite{dalcin2019conservative}, which will be used as a basis for the isothermal wall boundary condition derived in this work.

An SBP operator for the first derivative is defined as a matrix operator with constant coefficients $\mathcal{D} \in \mathbb{R}^{N \times N}$ that is a linear SBP operator of degree $p$ approximating the derivative $\frac{\partial  }{\partial   x_i}$ on the domain $x_i \in \left[ a,b\right]$ with nodal distribution $x$ having $N$ nodes, if
\begin{enumerate}
    \item $\mathcal{D}x^j_l=jx^{j-1}$, $j = 0,1,\cdots,p; $
    \item $\mathcal{D} = \mathcal{P}^{-1}\mathcal{Q}$, where the norm matrix $\mathcal{P}_{}$ is symmetric positive-definite;
    \item $\mathcal{Q}+\mathcal{Q}^{\top}=\mathcal{B}$, where $\mathcal{B}=\mathrm{diag}\left[-1,0,\cdots,0,1\right]$.
\end{enumerate}
See~\cite{fernandez2014review} for further details.

In addition to the LGL nodes, the scheme uses intermediate flux points defined by the $ p+1$-collocation complementary grid. These flux points form a complementary grid with spacing equal to the diagonal elements of the matrix $\mathcal{P}$, 
\begin{equation}
    \bar{x}_{i+1} - \bar{x}_i = \mathcal{P}_{ii}.
\end{equation}
Any 1D SBP differentiation operator $\mathcal{D}$ can be equivalently written in a telescoping form as follows.
\begin{equation}
    \mathcal{D}\mathbf{f} = \mathcal{P}^{-1} \Delta \bar{\mathbf{f}},
\end{equation}
where $\Delta$ is an $N_p \times (N_p+1)$ matrix corresponding to the two-point backward difference operator, and $\bar{\mathbf{f}}$ is $p$th-order flux vector defined at the flux points~\cite{fisher2013discretely,fisher2011boundary,carpenter2014entropy}. The one-dimensional SBP operators can be straightforwardly extended to multiple dimensions using tensor-product arithmetic \cite{carpenter2016entropy}. The multidimensional SBP operators defined in the computational domain are denoted with subscripts $\cdot_d$, where $d=1,2,3$. 


\subsection{Semi-discrete Navier--Stokes equations}

Using an SBP operator and its equivalent telescoping form, $\Delta_{x_I}$, in~\cite{parsani2014entropy}, the semi-discrete Navier--Stokes equations are written as follows.
\begin{equation}
\frac{\partial \mathbf{U}}{\partial t} 
= - \mathcal{P}^{-1}\Delta_{x_i}\bar{f}_i^{(I)}
+ \mathcal{D}_i [\mat{C}_{ij}] \Theta_{x_j}
+ \mathcal{P}^{-1} g^{U}_{x_i},
\end{equation}
where $\mat{C}_{ij}$ is the viscous flux matrix, and $g^{U}_{x_i}$ are the penalties associated with the conservative variables. $\Theta_{x_j}$ is the gradient of entropy variables computed using the procedure based on the local discontinuous Galerkin (LDG) method and internal penalties with an entropy gradient penalty developed in Dalcin 2019 \cite{dalcin2019conservative}, 
\begin{equation}
\Theta_{x_j}
= \mathcal{D}_{x_j} \mathbf{W} 
+ \mathcal{P}^{-1} g^{\Theta}_{x_j},
\end{equation}
where $g^{\Theta}_{x_j}$ is the penalty associated with the entropy variables. The penalties $g_{x_i}^\cdot$ are interchangable between boundary $g^{(B),\cdot}_{x_i}$ and internal $g^{(In),\cdot}_{x_i}$.

For high-speed viscous flows, the high-order scheme may suffer from the Gibbs phenomenon due to shocks and discontinuities. Thus, the positivity-preserving scheme developed in~\cite{upperman2023first,yamaleev2023high} is used in those cases. First, the entropy-stable high-order scheme is augmented with artificial dissipation based on the Brenner--Navier--Stokes equations~\cite{brenner2005navier,feireisl2010new}. This regularized Navier--Stokes model provides provable positivity of thermodynamic variables, the existence of weak solutions, and satisfaction of the entropy inequality for the Godunov entropy pair. The artificial dissipation is defined as an additional viscous flux term in the Navier--Stokes equations,
\begin{equation}
    F^{(AD)} = F^{(V)}(\mu_{AD}) + \sigma \frac{\partial \rho}{\partial x_i}\begin{bmatrix}
        1 \\ v \\ E
    \end{bmatrix},
\end{equation}
where $\mu_{AD}$ is defined as an entropy-residual-$S_n^k(R(U))$-scaled max viscosity $\mu_{AD} \sim |\Delta v| + \Delta \sqrt{p}$, and $\sigma$ is the volume diffusivity. The entropy residual detects spurious entropy production, distinguishes between different waves, and is design-order. This sensor only kicks in if the entropy production exceeds a certain tolerance. The resulting artificial dissipation can resolve jumps but would not be enough to avoid the Gibbs phenomenon, a byproduct of the high-order scheme. The high-order scheme will exhibit the Gibbs phenomenon and lead to undershoots in the solution. Thus, a first-order scheme is built on the same LGL points, and a hybrid scheme is used to switch between the high-order and first-order schemes based on the flux limiter $\theta_f$ as follows.
\begin{equation}
    \left(\frac{d\hat{\bm U}}{dt}\right)^n
    = \theta_f \left(\frac{d\hat{\bm U}}{dt}\right)^n_1
    + (1-\theta_f) \left(\frac{d\hat{\bm U}}{dt}\right)^n_p.
\end{equation}
The flux limiter $0 < \theta_f < 1$ is constant on a given high-order element and is nonzero only if any thermodynamic variable is negative. $\theta_f$ depends on the entropy residual $S_n^k$ to maintain pointwise positivity. This hybrid scheme is conservative, entropy-stable, satisfies the GCL condition, and is design-order. The full discussion of the positivity-preserving scheme is available in~\cite{upperman2021high,yamaleev2023high}.


\subsection{Entropy-stable wall boundary conditions}

Utilizing the common SAT procedure, the wall boundary conditions are imposed weakly via the penalty terms $g^{(B),U}_{x_i}$ and $g^{(B),\Theta}_{x_i}$. These entropy-stable boiundary conditions were first introduced in~\cite{parsani2014entropy}, providing the no-penetration wall used in this implementation. Later in~\cite{dalcin2019conservative} an LDG based approach was developed for the entropy gradient. In addition, these works extended the viscous wall boundary conditions to a moving adiabatic solid wall or a wall with a prescribed heat flux.

In the discussion of wall boundary conditions hereafter, we simplify the element to a $[0,1]^3$ cube with the wall boundary at $x_1=0$. Subsequently, the SBP operator $\mathcal{B}_{x_1}$ is decomposed into 
\begin{equation}
    \mathcal{B}_{x_1} = \mathcal{B}_{x_1}^- + \mathcal{B}_{x_1}^+
    = \text{diag}(-1,0,\cdots,0) + \text{diag}(0,0,\cdots,1).
\end{equation}
Thus, the operator $\mathcal{B}_{x_1}^-$ selects the wall boundary at $x_1=0$. Additionally, we omit the curvilinear metrics for simplicity, and assume that the wall is aligned with the $x_1$ direction. The following subsections summarize the wall boundary conditions implemented in Dalcin 2019 \cite{dalcin2019conservative}.

The boundary penalties implemented herein and based on\cite{dalcin2019conservative} are defined as follows:
\begin{itemize}
    \item $g^{(B),U}_{x_i}$ is the sum of the inviscid and viscous flux penalties, the dissipative matrix penalty, and the heat-entropy flux source term
    \begin{equation}
        g^{(B),U}_{x_i}
        = g^{(B,I)}_{x_i} 
        + g^{(B,V)}_{x_i} 
        + \mathcal{M}^{(B,V)}_{x_i} 
        + \mathcal{L}^{(B,V)}_{x_i},
    \end{equation}
    \item $g^{(B),\Theta}_{x_i}$ is the penalty for the entropy variables, and it penalizes the deviation from the entropy variables state at the wall.
    \begin{equation}
        g^{(B),\Theta}_{x_i}
        = \frac{1}{2}\mathcal{B}_{x_i}^-\left(
        \mathbf{W} - \mathbf{W}^\text{wall}
        \right),
    \end{equation}
    where the entropy variables at the wall $\mathbf{W}^\text{wall}$ is consistent with the viscous wall state $V^{V,wall}$ for the intended boundary condition.
\end{itemize}
Each of the terms in the boundary penalties is described individually next.

\textbf{1)} The no-penetration wall is enforced by setting the primitive wall state as
\begin{equation}
    V^{I,wall} = \text{diag}\left[1, -1, 1, 1, 1\right] V,
\end{equation}
and penalizing the inviscid flux against the entropy stable two-point flux for the internal state, $V$ and the wall state $V^{I,wall}$. The penalty is given by
\begin{equation}
    g^{(B,I)}_{x_1} = \mathcal{B}_{x_1}^- \left(
        f_{x_1} - f_{x_1}^{ssr}(V, V^{I,wall})
    \right),
\end{equation}
where $f_{x_1}^{ssr}$ is an entropy dissipative (stable) two-point entropy flux. In the current implementation, we use the Merriam flux \cite{merriam1989entropy}. If the wall is intended to enforce the symmetry conditions, a corresponding viscous wall state is constructed, and the penalty is applied to the viscous fluxes as well.

\textbf{2)} The no-slip or (inspirational) moving adiabatic wall is enforced using the no-penetration wall condition for the inviscid fluxes and further setting the viscous primitive wall state as
\begin{equation}
    V^{V,wall} = [\rho, -v_1, -v_2 + 2v_2^\text{wall}, -v_3 + 2v_3^\text{wall}, T]^\top,
\end{equation}
and further computing the gradient of the entropy variables by reflecting the primitive gradient density and temperature gradients,
\begin{equation}
    \Theta^\text{wall}_{x_j}
    = \frac{d \mathbf{W}}{d \mathbf{V}}\Bigr|_{V = V^{V,wall}} \Pi^\text{wall}_{x_j},
    \quad
    \Pi^\text{wall}_{x_j} = \text{diag}\left[-1, 1, 1, 1, -1\right] \Pi_{x_j},
    \quad
    \Pi_{x_j} = \frac{d \mathbf{W}}{d \mathbf{V}} \Theta_{x_j},
\end{equation}
where the final change-of-variables is computed using the primitive wall state $V^{V,wall}$. Finally, the no-slip or moving adiabatic wall is enforced via the penalty
\begin{equation}
    g^{(B,V)}_{x_1} = \frac{1}{2}\mathcal{B}_{x_1}^- \left(
        [\mat{C}_{1j}] \Theta_{x_j} - [\mat{C}_{1j}]^\text{wall} \Theta_{x_j}^\text{wall}
    \right),
\end{equation}
where $[\mat{C}_{1j}]^\text{wall}$ and $\Theta_{x_j}^\text{wall}$ are constructed using the primitive wall state $V^{V,wall}$. Additionally, this viscous wall state is used to compute the LDG penalty for the entropy gradient to compute the entropy variables at the wall $\mathbf{W}^\text{wall}$, $g^{(B),\Theta}_{x_1}$.

\textbf{3)} To further ensure the entropy stability of this scheme, an additional dissipative matrix penalty is added, using the viscous wall state $V^{V,wall}$ to compute the entropy variables at the wall $\mathbf{W}^\text{wall}$ in addition to the viscous matrix $[\mat{C}_{11}]^\text{wall}$, and the penalty is given by
\begin{equation}
    \mathcal{M}^{(B,V)}_{x_1} = -\beta 
    \frac{[\mat{C}_{11}] + [\mat{C}_{11}]^\text{wall}}{2} \left(
        \mathbf{W} - \mathbf{W}^\text{wall}
    \right),
\end{equation}
where $\beta$ scales with the inverse of the first element of the SBP norm matrix $\mathcal{P}$.

\textbf{4)} The final penalty is a heat entropy flux source term that applies to the energy equation defined for a heat flux $g(T)$ as
\begin{equation}
    \mathcal{L}^{(B,V)}_{x_1} 
    = \mathcal B_{x_1}^{-}[0,0,0,0,1]^\top g(T),
\end{equation}
where $g(T)$ is an $L^2$ function. Here, $g(T)$ can be chosen as any typical heat flux, such as the case presented in~\cite{dalcin2019conservative} for a wall with a prescribed heat flux.

\begin{note}
    The representation of $g(t)$ in~\cite{dalcin2019conservative} assumes a heat-entropy-flux absorbing the $1/T$ term, and thus, the heat-entropy-flux is defined as $g(t) = q^\text{wall}/T$. In this work, we define the heat flux as $g(T) = q^\text{wall}$, and thus, the $1/T$ term is explicitly present in the entropy balance. This allows for a more general definition of the heat flux at the wall.
\end{note}


\subsection{The isothermal wall boundary condition}
\label{subsec:isothermal-wall-heat-flux}

The present approach imposes the penalty as a heat-flux consistent with the isothermal boundary condition. The heat-flux correction term $g(t)$ is applied as a source term that results in a penalty in the energy equation as follows. 
\begin{align*}
    \mathcal{L}^{(B,V)}_{x_1} 
    = \mathcal{B}_{x_1}^- [0,0,0,0,1]^\top g(t)
    &= \mathcal{B}_{x_1}^- [0,0,0,0,1]^\top \left(-\kappa \alpha(T - T^\text{wall})\right).
\end{align*}

\begin{remark}
    The term $\alpha$ is a penalty parameter that can be tuned to ensure that the penalty does not exceed the physical heat-entropy flux at the wall. To enforce the design-order properties of the scheme, $\alpha$ is set to scale with the inverse of the grid spacing, $\Delta x$, to the power of the polynomial degree $p$. However, due to the increased stiffness, $\alpha$ can be limited to a smaller value by choosing the average spacing $\Delta x_\text{avg}$ and/or capping the polynomial power to a smaller integer $q$.
\end{remark}


\subsubsection{Contribution to the semi-discrete entropy balance}

SBP operators allow for mimicking the integration-by-parts property at the discrete level, resulting in a discrete entropy balance that mimics the continuous entropy inequality. By multiplying the semi-discrete Navier--Stokes equations by the entropy variables $\mathbf{W}$, and integrating using the quadrature matrix $\mathcal{P}$, the semi-discrete entropy balance is derived as follows.
\begin{align*}
    W^\top \mathcal{P} \frac{\partial U}{\partial t}
    &= - W^\top \mathcal{P} \mathcal{P}^{-1} \Delta_{x_i} \bar{f}_i^{(I)}\\
    &+ W^\top \mathcal{P} \mathcal{D}_i [\mat{C}_{ij}] \Theta_{x_j}\\
    &+ W^\top \mathcal{P} \mathcal{P}^{-1} g^{U}_{x_i}
\end{align*}
The first term on the right-hand side is the inviscid contribution to the entropy balance, and the last term is the conservative penalty contribution on the element surface. The second term is of interest and can be further decomposed as follows.
\begin{align*}
    W^\top \mathcal{P} \mathcal{D}_i &[\mat{C}_{ij}] \Theta_{x_j}
    = W^\top \mathcal{P} \mathcal{P}^{-1}_i \mathcal{Q}_i [\mat{C}_{ij}] \Theta_{x_j} \\
    &= W^\top \mathcal{P} \mathcal{P}^{-1}_i \mathcal{Q}_i [\mat{C}_{ij}] \Theta_{x_j}
    + W^\top \mathcal{P} \mathcal{P}^{-1}_i \mathcal{Q}^\top_i [\mat{C}_{ij}] \Theta_{x_j}
    - W^\top \mathcal{P} \mathcal{P}^{-1}_i \mathcal{Q}^\top_i [\mat{C}_{ij}] \Theta_{x_j}\\
    &= W^\top \mathcal{P} \mathcal{P}^{-1}_i \mathcal{B}_i [\mat{C}_{ij}] \Theta_{x_j}
    - W^\top \mathcal{P} \mathcal{P}^{-1}_i \mathcal{Q}^\top_i [\mat{C}_{ij}] \Theta_{x_j}
\end{align*}
The first term in this relation is the viscous contribution to the entropy balance on the surface and includes the wall heat flux, which will be discussed in the next paragraph. The latter term can be further decomposed as follows.
\begin{align*}
    W^\top \mathcal{P} \mathcal{P}^{-1}_i \mathcal{Q}^\top_i &[\mat{C}_{ij}] \Theta_{x_j}\\
    &= \left([\mat{C}_{ij}] \Theta_{x_j}\right)^\top \mathcal{P} \mathcal{D}_i W
    + \left([\mat{C}_{ij}] \Theta_{x_j}\right)^\top \mathcal{P} \mathcal{P}^{-1}_i g^{\Theta}_{x_i}
    - \left([\mat{C}_{ij}] \Theta_{x_j}\right)^\top \mathcal{P} \mathcal{P}^{-1}_i g^{\Theta}_{x_i}\\
    &= \left([\mat{C}_{ij}] \Theta_{x_j}\right)^\top \mathcal{P} 
    \left(\mathcal{D}_i W + \mathcal{P}^{-1}_i g^{\Theta}_{x_i}\right)
    - \left([\mat{C}_{ij}] \Theta_{x_j}\right)^\top \mathcal{P} \mathcal{P}^{-1}_i g^{\Theta}_{x_i}\\
    &= \left([\mat{C}_{ij}] \Theta_{x_j}\right)^\top \mathcal{P} \Theta_{x_i}
    - \Theta_{x_j}^\top \mathcal{P} [\mat{C}_{ij}] \mathcal{P}^{-1}_i g^{\Theta}_{x_i}
\end{align*}
The first term is positive-definite and represents the viscous dissipation in the entropy balance, while the second term is the contribution from the entropy gradient penalty. The resulting entropy balance is given as follows.
\begin{align*}
    1^\top \mathcal{P} \frac{\partial \mathcal{S}}{\partial t} 
    + \Theta_{x_j}^\top \mathcal{P} [\mat{C}_{ij}] \Theta_{x_j}
    &= 1^\top \mathcal{P} \mathcal{P}^{-1} \mathcal{B}_{x_i} \mathcal{F}_i
    && \text{(Inviscid)}\\
    &+ \mathbf{W}^\top \mathcal{P} \mathcal{P}^{-1} \mathcal{B}_{x_i} [\mat{C}_{ij}] \Theta_{x_j}
    && \text{(Viscous)}\\
    &+ \mathbf{W}^\top \mathcal{P} \mathcal{P}^{-1} g^{U}_{x_i}
    && \text{(Conservative penalty)}\\
    &+ \left([\mat{C}_{ij}] \Theta_{x_j}\right)^\top \mathcal{P} \mathcal{P}^{-1} g^{\Theta}_{x_i},
    && \text{(Entropy gradient penalty)}
\end{align*}
where the first term is due to the entropy-consistent flux satisfying the relation $\Delta_{x_i} \bar{f}_i^{(I)} = \mathcal{B}_{x_i} \mathcal{F}_i$~\cite{fisher2013discretely}.

To reveal the wall heat flux, the viscous matrix is decomposed into velocity and temperature gradient components, respectively. 
\begin{equation}
    [C_{ij}] 
    = \begin{bmatrix}
        0 & 0 & 0 & 0 & 0\\
        0 & \frac43\mu T & 0 & 0 & \frac43\mu T v_1\\
        0 & 0 & \mu T & 0 & \mu T v_2\\
        0 & 0 & 0 & \mu T & \mu T v_3\\
        0 & \frac43\mu T v_1 & \mu T v_2 & \mu T v_3 & \mu T (\frac43v_1^2 + v_2^2 + v_3^2)
    \end{bmatrix}
    + \begin{bmatrix}
        0 & 0 & 0 & 0 & 0\\
        0 & 0 & 0 & 0 & 0\\
        0 & 0 & 0 & 0 & 0\\
        0 & 0 & 0 & 0 & 0\\
        0 & 0 & 0 & 0 & \kappa T^2
    \end{bmatrix}.
\end{equation}
Thus, the contribution of the viscous terms to the entropy balance is computed as follows.
\begin{align*}
    W^\top \mathcal{P} \mathcal{P}^{-1}_i \mathcal{B}_i [\mat{C}_{ij}] \Theta_{x_j}
    &= W^\top \mathcal{P} \mathcal{P}^{-1}_i \mathcal{B}_i \Lambda
    + W^\top \mathcal{P} \mathcal{P}^{-1}_i \mathcal{B}_i [0, 0, 0, 0, 1]^\top \kappa \frac{\partial T}{\partial x_i}
\end{align*}
where $\Lambda$ is the contribution of the velocity gradients to the viscous flux. To simplify the display, let $L^-= \mathcal{P} \mathcal{P}^{-1}_i \mathcal{B}_i^- [0, 0, 0, 0, 1]^\top$. Taking the second term on the $x_i=0$ face and adding the defined penalty, the desired isothermal wall heat flux id obtained as follows.
\begin{equation}
    \label{eq:semidiscrete_dSdt_isothermal}
      W^\top L^- \kappa \frac{\partial T}{\partial x_i}
      + W^\top L^- g(t)
      = W^\top L^- \left(\kappa \frac{\partial T}{\partial x_i} - \kappa \alpha (T - T^\text{wall})\right)
\end{equation}
Therefore, the resulting entropy balance on the wall face is written as follows.
\begin{align*}
    1^\top \mathcal{P} \frac{\partial \mathcal{S}}{\partial t} 
    &+ \Theta_{x_j}^\top \mathcal{P} [\mat{C}_{ij}] \Theta_{x_j}
    = 1^\top \mathcal{P} \mathcal{P}^{-1} \mathcal{B}_{x_i}^- \mathcal{F}_i
    && \text{(Inviscid)}\\
    &+ \mathbf{W}^\top \mathcal{P} \mathcal{P}^{-1} \mathcal{B}_{x_i}^- \Lambda
    && \text{(Viscous w/o heat-flux)}\\
    &+ \mathbf{W}^\top \mathcal{P} \mathcal{P}^{-1} \left(
        g^{(B,I)}_{x_i} 
      + g^{(B,V)}_{x_i} 
      + \mathcal{M}^{(B,V)}_{x_i}\right)
    && \text{(Conservative penalties)}\\
    &+ \left([\mat{C}_{ij}] \Theta_{x_j}\right)^\top \mathcal{P} \mathcal{P}^{-1} g^{\Theta}_{x_i}.
    && \text{(Entropy gradient penalty)}\\
    &+ W^\top L^- \kappa \left(\kappa \frac{\partial T}{\partial x_i} - \kappa \alpha (T - T^\text{wall})\right)
    && \text{(Isothermal wall heat-flux)}\\
\end{align*}
Upon convergence, up to tolerance, of the wall nodes to the no-slip (or moving) isothermal wall, 
\begin{equation}
    v_i = 0, \quad\quad v_{j\neq i} = v^\text{wall}_{j\neq i} \quad\quad T = T^\text{wall},
\end{equation}
all the surface terms vanish except for $\frac{\partial T}{\partial x_i}$, since it generally does not vanish for the isothermal condition, thus the resulting entropy balance is given as follows.
\begin{align*}
    1^\top \mathcal{P} \frac{\partial \mathcal{S}}{\partial t} 
    + \Theta_{x_j}^\top \mathcal{P} [\mat{C}_{ij}] \Theta_{x_j}
    &= -\mathbf{1}^\top L^- \kappa \frac{1}{T} \frac{\partial T}{\partial x_i},
\end{align*}
a discrete analog to the continuous entropy balance.


\subsubsection{Contribution to the semi-discrete ballistic energy}

Similar to the entropy balance, the semi-discrete ballistic energy balance can be derived using SBP operators to achieve a desirable bound on the solution. The semi-discrete ballistic energy rate on a given element is expressed as follows.
\begin{equation}
  \mathbf{1}^\top \mathcal{P} \frac{\partial B}{\partial t}
  = [0, 0, 0, 0, 1] \mathcal{P} \frac{\partial (\rho E)}{\partial t}
  + {\mathbf{T}^\text{wall}}^\top \mathcal{P} \frac{\partial S}{\partial t}.
\end{equation}
The following Lemma establishes the semi-discrete counterpart bound.

\begin{lemma}[Semi-discrete ballistic rate bound]
    \label{lem:semi-discrete-ballistic-rate-bound}
    Assuming the positivity of the thermodynamic variables and the wall boundary conditions are satisfied up to design-order except for the isothermal condition,
    \begin{equation}
        v_i = 0, \quad\quad v_{j\neq i} = v^\text{wall}_{j\neq i} \quad\quad T \neq T^\text{wall},
    \end{equation}
    the semi-discrete ballistic energy is bounded as follows.
    \begin{equation}
        \label{eq:semi-discrete-ballistic-rate-bound}
        \mathbf{1}^\top \mathcal{P} \frac{\partial B}{\partial t}
        \leq \mathbf{1}^\top L^-
        \left(
            \kappa\frac{1}{4 \alpha T} \left(\frac{\partial T}{\partial n}\right)^2
        \right).
    \end{equation}
\end{lemma}
\begin{proof}
    Equation~\eqref{eq:semidiscrete_dSdt_isothermal}, and the added penalty term to the energy equation provide,
    \begin{equation*}
        \mathbf{1}^\top \mathcal{P} \frac{\partial B}{\partial t}
        \leq \mathbf{1}^\top L^-
        \left(
            \kappa \frac{\partial T}{\partial x_i} - \kappa \alpha (T - T^\text{wall})
            - \kappa \frac{T^\text{wall}}{T} \frac{\partial T}{\partial x_i} + \kappa \alpha T^\text{wall} \frac{T - T^\text{wall}}{T}
        \right).
    \end{equation*}
    The remaining algebra is the same as in Lemma~\ref{lem:ballistic-rate-bound}, resulting in the following semi-discrete bound.
    \begin{equation}
        \mathbf{1}^\top \mathcal{P} \frac{\partial B}{\partial t}
        \leq \mathbf{1}^\top L^-
        \left(
            -\kappa\frac{1}{T}
            \left[
                \left(
                    - \frac{1}{2\sqrt{\alpha}} \frac{\partial T}{\partial n} 
                    + \sqrt{\alpha} (T - T^\text{wall})
                \right)^2
                - \frac{1}{4\alpha} \left(\frac{\partial T}{\partial n}\right)^2
            \right]
        \right).
    \end{equation}
    Distributing the negative signs and absorbing the non-positive term under the inequality provides the desired bound~\eqref{eq:semi-discrete-ballistic-rate-bound}.
\end{proof}

The assumption of positivity of thermodynamic variables is possible numerically as well, provided that a positivity-preserving scheme is used. In this paper, the positivity-preserving schemes in~\cite{yamaleev2023high,upperman2023first} are used in cases prone to positivity violations with the baseline schemes.

The SBP-SAT framework allows for mimicking the continuous integration and differentiation properties. Using these operators, a semi-discrete version of the result in Theorem~\ref{thm:ballistic-energy-bound} is presented next for the affine entropy used for analysis purposes only.

\begin{theorem}[Semi-discrete solution $L^2$ bound]
    \label{thm:semi-discrete-solution-l2-bound}
    Assume the positivity of thermodynamic variables under a positivity-preserving scheme. If the boundary term,
    \begin{equation}
        \mathcal{C}_\text{ballistic}(\alpha)
        = \mathbf{1}^\top L^-
        \left(
            \frac{\kappa}{4\alpha T} \left(\frac{\partial T}{\partial n}\right)^2
            + \kappa \frac{T^\text{wall}}{T_0}\frac{\partial T}{\partial n}
        \right)
        < \infty,
    \end{equation}
    is bounded, then the semi-discrete solution of the compressible Navier--Stokes equations with the isothermal wall boundary condition is bounded in the $L^2$ norm by initial and boundary data
    \begin{align}\label{eq:semi-discrete-solution-bound}
        \begin{split}
            \|U(\mathcal{T})\|_{\mathcal{P}}^2
            &\leq
            \frac{2}{\lambda_\text{min}\left(\frac{d^2 S(\bar(U))}{d U^2}\right)}
            \left[[0, 0, 0, 0, 1] \mathcal{P} \rho E(0) + \int_0^\mathcal{T} \mathcal{C}_\text{ballistic}(\alpha)\ dt\right]\\
            &+\|U(0)\|_{\mathcal{P}}^2 + \mathcal{T} \mathbf{1}^\top L^- \left(\kappa \alpha \frac{(T^\text{wall})^2}{T0}\right)\\
        \end{split}
    \end{align}
\end{theorem}

\begin{proof}
    Similar to the proof of Theorem~\ref{thm:ballistic-energy-bound}, the proof proceeds in two directions. First, the ballistic energy rate bound is computed for the affine entropy $\bar{\mathcal{S}}$. Lemma~\ref{lem:semi-discrete-ballistic-rate-bound} provides the bound for the Godunov entropy. Similarly, the constant terms in the affine entropy are canceled by the time derivative. The remaining term is bounded as follows.
    \begin{equation}
        \frac{d \mathcal{S}(0)}{d U}^\top \mathcal{P} \mathcal{P}^{-1}_i \frac{\partial U}{\partial t}
        \leq \mathbf{1}^\top L^- \left(
            \kappa \frac{T^\text{wall}}{T_0} \frac{\partial T}{\partial n}
            - \kappa \alpha \frac{T^\text{wall}}{T_0} (T - T^\text{wall})
        \right).
    \end{equation}
    The first term is the second term in $C_\text{ballistic}(\alpha)$. A positivity-preserving scheme provides $T > 0$, so the second term in the previous expression is bounded by a constant, as done in the proof of Theorem~\ref{thm:ballistic-energy-bound}. In addition, a positivity-preserving scheme also provides $\rho E > 0$ and thus, similar to the continuos proof,
    \begin{equation}
        \bar{\mathcal{S}}(\mathcal{T}) \leq \frac{B}{T^\text{wall}}
    \end{equation}

    Lemma~\ref{lem:entropy_bounds_solution_diffrence} applies semi-discretely as well, since the bound is obtained via a Taylor's theorem for the affine entropy. The SBP operator $\mathcal{P}$ provides a norm, thus the triangle inequality holds.
    \begin{equation}\label{eq:semi-discrete-triangle-inequality}
      \|U(\mathcal{T}) - U(0) + U(0)\|_{\mathcal{P}}^2
      \leq \|U(\mathcal{T}) - U(0)\|_{\mathcal{P}}^2
      + \|U(0)\|_{\mathcal{P}}^2.
    \end{equation}
    Combining the inequality from the solution difference, the semi-discrete ballistic energy bound for the affine transformed entropy, and the semi-discrete triangle inequality~\eqref{eq:semi-discrete-triangle-inequality}, the desired bound on the discrete solution is obtained.
\end{proof}

This completes the building blocks for an entropy-stable high-order semi-discrte scheme with isothermal boundary conditions.


\section{Numerical results}
\label{sec:numerical-results}


All cases used meshes generated with Gmsh~\cite{geuzaine2009gmsh}. The visualizations were exported from Tecplot. Some line profile plots were generated using the Matplotlib package~\cite{hunter2007matplotlib}. The first two cases utilize the Portable, Extensible Toolkit for Scientific Computation (PETSc) time-stepping functionality~\cite{petsc-ts}. The semi-discrete high-order entropy-stable scheme detailed in the previous sections is implemented as an in-house solver and is used to obtain the results in this section.

For all test problems considered, the flow quantities are non-dimensionalized as follows: $t=\frac{t'}{L^*/v^*}$, $x_j=\frac{x_j'}{L^*}$, $\rho=\frac{\rho'}{\rho^*}$, $v_i=\frac{v_i'}{v^*}$, $\mu=\frac{\mu'}{\mu^*}$, $p=\frac{p'}{p^*}$, $R=\frac{R'}{R^*}$, $E=\frac{E'}{\gamma c_{\rm v}^* T^*}$, $\kappa=\frac{\kappa'}{\kappa^*}$ and $c_{\rm v}=\frac{c_{\rm v}'}{c_{\rm v}^*}$. Thus, the non-dimensional pressure, total and kinetic energy variables are given by
\begin{align}
    &p=\rho T, &E=\frac{1}{\gamma} T + E_k, &&E_k=\frac{(\gamma-1)M_{\infty}^2}{2}u^2.
\end{align}
The exponential nature of the parameter $\alpha$ ensures that the influence of the penalty term scales appropriately with the mesh refinement. This setup significantly increases the stiffness of the system, which may necessitate smaller time steps for explicit time integration schemes. Thus, the implicit backward euler was used to alleviate the stiffness in cases using $\alpha=\Delta x^{-p}$. The error calculations are defined as follows.
\begin{align*}
    e_\text{temperature}&=\|T-T^\text{wall}\|_{L^2}/\|\Gamma\|_{L^2},\\
    e_\text{velocity}&=\|v-v^\text{wall}\|_{L^2}/\|\Gamma\|_{L^2},\\
    BE_{\text{est},L2}&=\|BE_\text{est}\|_{L^2}/\|\Gamma\|_{L^2},
\end{align*}
where $\Gamma$ is the wall boundary length, $T^\text{wall}$ is the isothermal boundary value at the face node, $v^\text{wall}$ is the tangential wall velocity (inspirational moving wall condition, and $BE_\text{est}$ is the expression in~\eqref{eq:ballistic-energy-bound}. Herein, four cases are considered to show the numerical properties of this penalty.
\begin{itemize}
    \item a 2D channel flow to verify the design-order accuracy of the proposed isothermal wall boundary condition,
    \item a 2D flow over a square cylinder to further verify the design-order accuracy,
    \item a supersonic flow over a cylinder to demonstrate the ability of the proposed boundary condition to handle extreme heat fluxes at the wall, and
    \item a thermal Couette flow to demonstrate the accuracy in comparison to the available reference case.
\end{itemize}


\subsection{2D Channel flow}

A 2D channel flow with periodic, dissipating, sinusoidal waves is considered to demonstrate the design-order of the isothermal and no-slip wall boundary conditions. In addition, the ballistic energy bound~\eqref{eq:ballistic-energy-bound} is measured to verify whether the bound vanishes. The 2D channel flow described in~\cite{chan2022entropy} spans $[-2,2,]\times[-1,1]$ subdivided as $2N\times N$ with mesh spacing in the $y$ direction defined by
\begin{equation}
    \tilde{y} = y_\text{uniform} + \frac{1}{4}\sin(\pi y_\text{uniform}).
\end{equation}

The ideal gas flow is characterized by $Re=50$, $Pr=0.75$, and $Ma=0.3$, the reference temperature is $T_\text{ref}=300 K$, and the wall temperature is $T_w=310 K$. A noslip boundary condition is applied at the walls, and periodic boundary conditions are applied in the streamwise direction. The case is initialized by the following initial conditions
\begin{equation}
    \rho = 1, 
    \quad u 
    = \frac{1}{10}
    \sin\left(\frac{\pi x}{2}\right)
    \cos\left(\frac{\pi y}{2}\right), 
    \quad v = \frac{1}{10}
    \cos\left(\frac{\pi x}{2}\right)
    \sin(\pi y), 
    \quad P = 1.
\end{equation}
The case is advanced until time $t=0.5$ using the implicit backward Euler method. The convergence of the temperature, velocity, and ballistic energy bound at the wall are assessed for $p=2,3,4$ on a sequence of uniformly refined meshes. The resulting flow at $t=0.5$ is shown in Figure \ref{fig:chan_flow}. The temperature profile displays an expected flare towards the walls due to the isothermal boundary condition being higher than the flow temperature.

\begin{figure}
    \begin{minipage}{0.49\textwidth}
        \includegraphics[width=\textwidth]{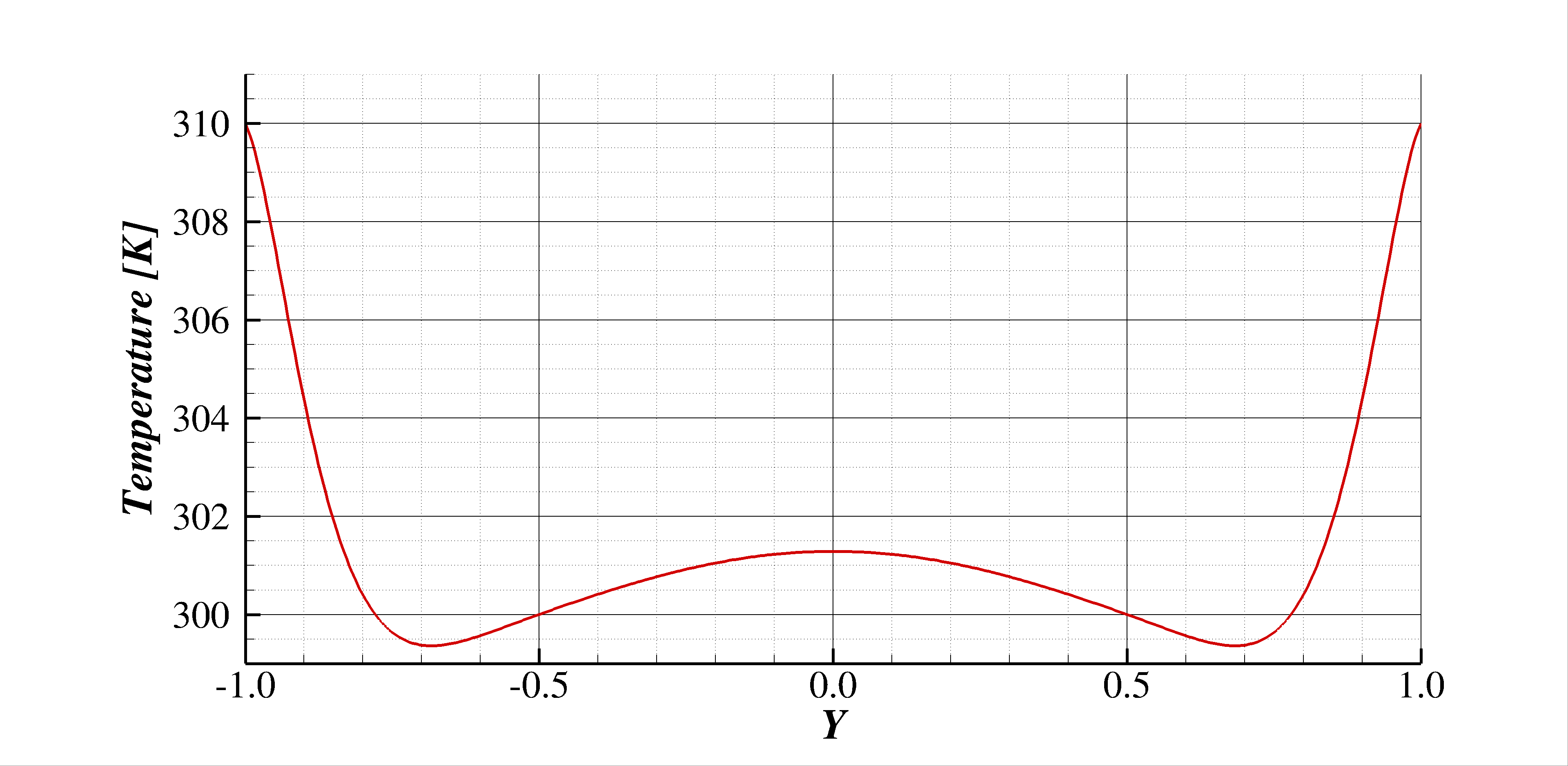}
    \end{minipage}
    \begin{minipage}{0.49\textwidth}
        \includegraphics[width=\textwidth]{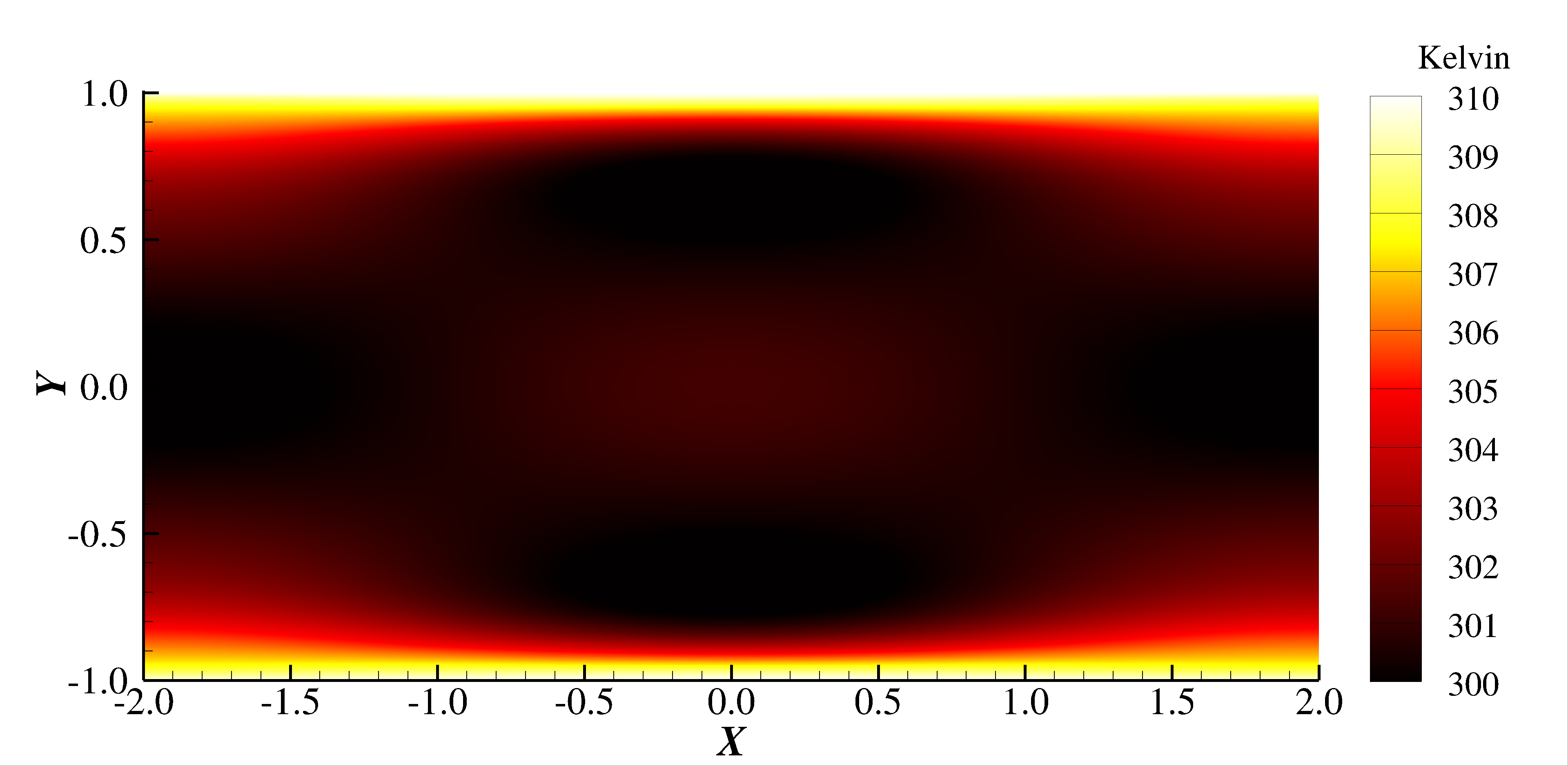}    
    \end{minipage}
    \caption{Temperature field at $t=0.5$ for the channel case with $N=64$ and $p=4$.}
    \label{fig:chan_flow}
\end{figure}

The convergence results for the wall temperature are summarized in Table \ref{tab:chan_temperature}. The ideal convergence rate of at least $\mathcal{O}(h^{p+1})$ is observed for all cases $p=2,3,4$. This shows that the scheme is design-order accurate for the temperature at the wall.

\begin{table}
    \caption{Wall temperature convergence rates of the channel case with $p=2,3,4$ on a globally refined uniform grid.}
    \label{tab:chan_temperature}
    \centering
    \begin{tabular}{c|cccccc}
        $N_{elem}$ & $e_\text{temperature}$ & $p=2$ & $e_\text{temperature}$ & $p=3$ & $e_\text{temperature}$ & $p=4$ \\
        \hline
        $4$  & $7.81\times 10^{-4}$ &   --   & $8.81\times 10^{-5}$  &   --   & $1.19\times 10^{-5}$  &   --   \\
        $8$  & $3.33\times 10^{-5}$ & $4.55$ & $1.12\times 10^{-6}$  & $6.29$ & $4.38\times 10^{-8}$  & $8.08$ \\
        $16$ & $2.74\times 10^{-6}$ & $3.60$ & $3.62\times 10^{-8}$  & $4.96$ & $5.63\times 10^{-10}$ & $6.28$ \\
        $32$ & $3.00\times 10^{-7}$ & $3.19$ & $1.85\times 10^{-9}$  & $4.29$ & $1.44\times 10^{-11}$ & $5.29$ \\
        $64$ & $3.62\times 10^{-8}$ & $3.05$ & $1.23\times 10^{-10}$ & $3.92$ & $4.40\times 10^{-13}$ & $5.03$ \\
    \end{tabular}
\end{table}

For the wall velocity, summarized in Table~\ref{tab:chan_velocity}, the expected convergence rate of $\mathcal{O}(h^{p+1})$ is observed for all polynomial orders considered, consistent with the results reported in~\cite{chan2022entropy}. This demonstrates that the novel isothermal wall boundary condition does not affect the convergence of the velocity at the wall.

\begin{table}
    \caption{Wall velocity convergence rates of the channel case with $p=2,3,4$ on a globally refined uniform grid.}
    \label{tab:chan_velocity}
    \centering
    \begin{tabular}{c|cccccc}
        $N_{elem}$ & $e_\text{velocity}$ & $p=2$ & $e_\text{velocity}$ & $p=3$ & $e_\text{velocity}$ & $p=4$ \\
        \hline
        $4$  & $4.71\times 10^{-2}$ &   --   & $1.14\times 10^{-2}$ &   --   & $2.35\times 10^{-4}$  &   --   \\
        $8$  & $1.72\times 10^{-3}$ & $4.77$ & $1.59\times 10^{-3}$ & $2.84$ & $4.07\times 10^{-4}$  & $-0.8$ \\
        $16$ & $3.88\times 10^{-4}$ & $2.15$ & $1.41\times 10^{-4}$ & $3.49$ & $1.13\times 10^{-5}$  & $5.17$ \\
        $32$ & $1.17\times 10^{-5}$ & $5.06$ & $1.10\times 10^{-5}$ & $3.67$ & $1.32\times 10^{-8}$  & $9.74$ \\
        $64$ & $1.71\times 10^{-6}$ & $2.77$ & $3.08\times 10^{-7}$ & $5.17$ & $2.85\times 10^{-10}$ & $5.53$ \\
    \end{tabular}
\end{table}

In addition to the wall boundary conditions, the bound on the ballistic energy \eqref{eq:ballistic-energy-bound} is assessed using the $L^2$ domain normalized norm $BE_{\text{est},L2}$. The convergence of the ballistic energy bound is summarized in Figure~\ref{fig:chan_ballistic}. This estimate does not necessarily need to be of design-order; more importantly, it shows that, while the temperature gradient at the wall is not known a priori, the ballistic energy bound converges to zero with mesh refinement.

\begin{figure}
    \centering
    \begin{minipage}{0.49\textwidth}
        \includegraphics[width=\textwidth]{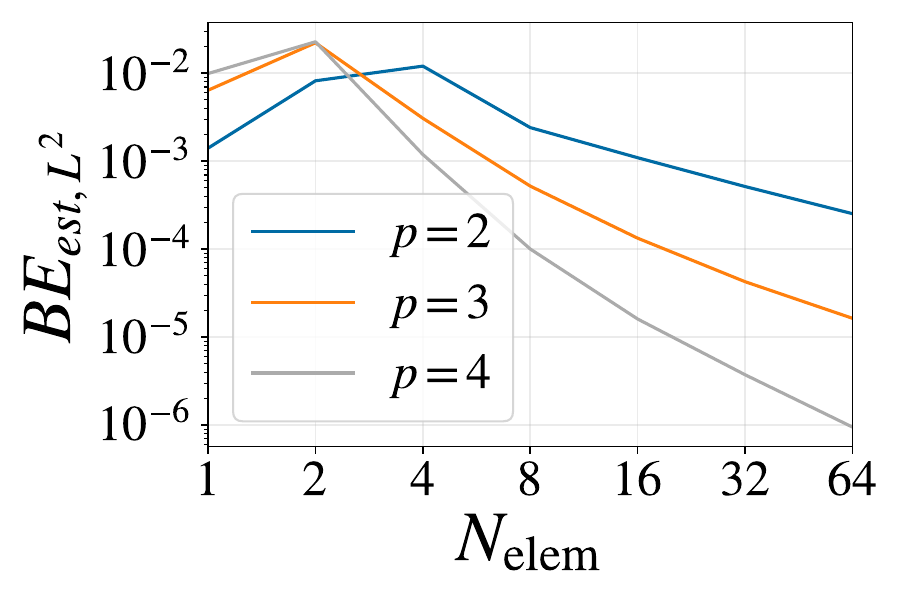}
    \end{minipage}
    \hfill
    \begin{minipage}{0.49\textwidth}
        \vspace{6em}
        \caption{Wall ballistic energy bound convergence rates of the channel case with $p=2,3,4$ on a globally refined uniform grid.}
        \label{fig:chan_ballistic}
    \end{minipage}
\end{figure}


\subsection{Subsonic flow around a cooling square cylinder}
 
The case of a flow around a cooling square is considered based on the case in~\cite{parsani2014entropy}. The ideal gas flow is characterized by $Re=200$, $Pr=0.72$, and $Ma=0.3$, the reference temperature is $T_\text{ref}=200 K$, and the wall temperature is $T^\text{wall}=150 K$. The boundary conditions are no-slip and isothermal at the wall, uniform inflow at the left boundary, and inviscid (slip) symmetry along the top and bottom boundaries, and a subsonic outflow at the right boundary. The case is initialized by a uniform flow in the whole domain. 

The mesh is a two-dimensional mesh with a square of length $d$ centered at the origin. The bounding rectangle is $[-23.5d,57.5d]\times[-29.5d,29.5d]$. The mesh contains a structured area surrounding the square and in the wake. The refinement doubles the boundary layer elements, $N_\text{elem}$, and proportionally the rest of the structured regions, resulting in approximatly a quadruple increase in total elements per refinement. The meshing algorithm results in a full quad mesh.

The case is advanced until time $t=5$ using the backward Euler method, and the convergence of the wall temperature is assessed at the final time. The results in Table~\ref{tab:cyl_convergence} presents the convergence rates where the parameter $\alpha=\Delta x^{-p}$ is set as defined in section~\ref{subsec:isothermal-wall-heat-flux}. The wall temperature errors converge at the expected $\mathcal{O}(h^{p+1})$ for both orders of accuracy considered, $p=2,3$. Therefore, the proposed isothermal penalty is design-order for wall boundaries with sharp edges.

\begin{table}
    \caption{Wall temperature convergence rates of the cylinder case with $p=2,3$ on a globally refined grid. The parameter $\alpha=\Delta x^{-p}$ is used in the simulations.}
    \label{tab:cyl_convergence}
    \centering
    \begin{tabular}{c|cccc}
        $N_{elem}$ & $e_\text{temperature}$ & $p=2$ & $e_\text{temperature}$ & $p=3$ \\
        \hline
        $1$ & $7.74\times 10^{-4}$ &   --   & $3.08\times 10^{-5}$ &   --   \\
        $2$ & $1.01\times 10^{-4}$ & $2.94$ & $2.14\times 10^{-6}$ & $3.85$ \\
        $4$ & $1.06\times 10^{-5}$ & $3.25$ & $1.15\times 10^{-7}$ & $4.22$ \\
        $8$ & $1.29\times 10^{-6}$ & $3.03$ & $7.13\times 10^{-9}$ & $4.01$ \\
    \end{tabular}
\end{table}

The computed ballistic energy bound is presented in Figure~\ref{fig:cyl_be_est}. Similar to the channel case, the ballistic energy estimate converges to zero with grid refinement. This shows that the ballistic energy rate for this case is bounded and, thus entropy-stable.

\begin{figure}
    \centering
    \begin{minipage}{0.49\textwidth}
        \includegraphics[width=\textwidth]{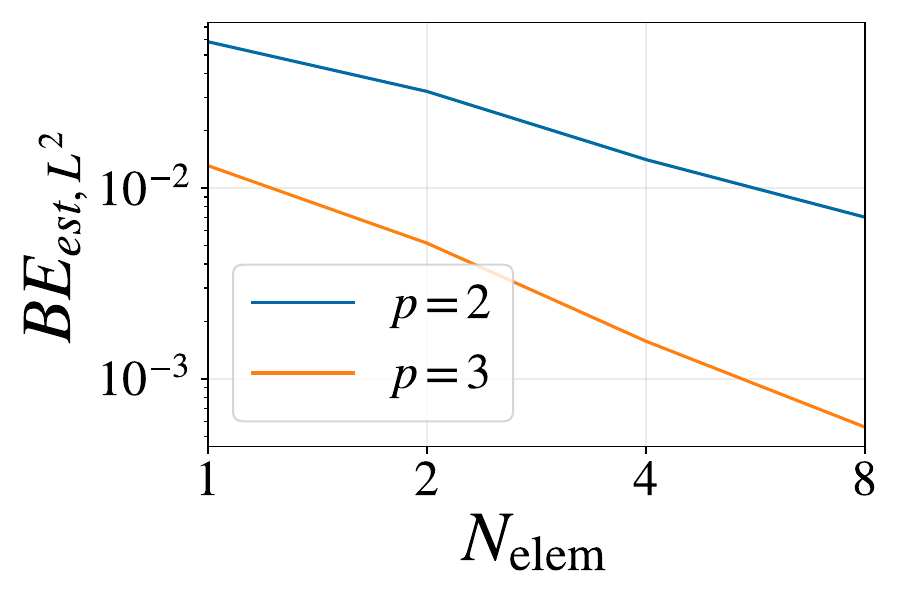}
    \end{minipage}
    \hfill
    \begin{minipage}{0.49\textwidth}
        \vspace{6em}
        \caption{Wall ballistic energy bound convergence rates of the square cylinder case with $p=2,3$ on a globally refined grid.}
        \label{fig:cyl_be_est}
    \end{minipage}
\end{figure}


\subsection{Bow shock ahead of a heated cylinder}
\label{subsec:bow-shock-heated-cylinder}

The case of supersonic argon flow over a cylinder is considered to showcase the ability to enforce the isothermal condition for cooling extreme heat at the stagnation point and the side of the cylinder. This case is adapted from~\cite{peng2024extensions} where the flow is characterized by $\gamma=1.667$, $Re=5000$, $Pr=0.667$ and $Ma=10$, the reference temperature is $T_\text{ref}=200 K$, and the wall temperature is $T^\text{wall}=500 K$. The boundary conditions are no-slip and isothermal at the wall, uniform inflow at the left boundary, and supersonic outflow at the top, right, and bottom boundaries. The case is initialized with a uniform flow at the inflow conditions, which transitions to the Dirichlet values with a hyperbolic tangent profile for the velocity and temperature near the cylinder to avoid a large initial mismatch. 

The computational domain is two-dimensional and consists of a circular cylinder of diameter $1$ centered at the origin, within the rectangular domain $[-2,3]\times[-2,2]$. A thin near-wall boundary-layer region of thickness $8\times10^{-3}$ surrounds the cylinder, with $3$ elements across the layer and $72$ elements per quadrant along the cylinder surface. The mesh is stretched away from the cylinder. The mesh contains $15,192$ quadrilateral elements.

In this $Ma=10$ case, the Reynolds number yields $Kn=0.002$ as done in~\cite{peng2024extensions}. This case is particularly challenging due to the high Mach number, which increases the temperature at the stagnation upto $\sim 4200K$. The case uses a $p=4$ polynomial approximation and is advanced until $t=10$. The positivity-preserving SSPRK3 time-stepper based on~\cite{yamaleev2023high} was used to ensure the positivity of thermodynamic variables. Due to the stiffness of this case, the isothermal penalty parameter is set to $\alpha=\Delta x_\text{avg}^{-2}$.

\begin{figure}
  \centering
  \begin{minipage}{0.49\textwidth}
    \includegraphics[width=\textwidth]{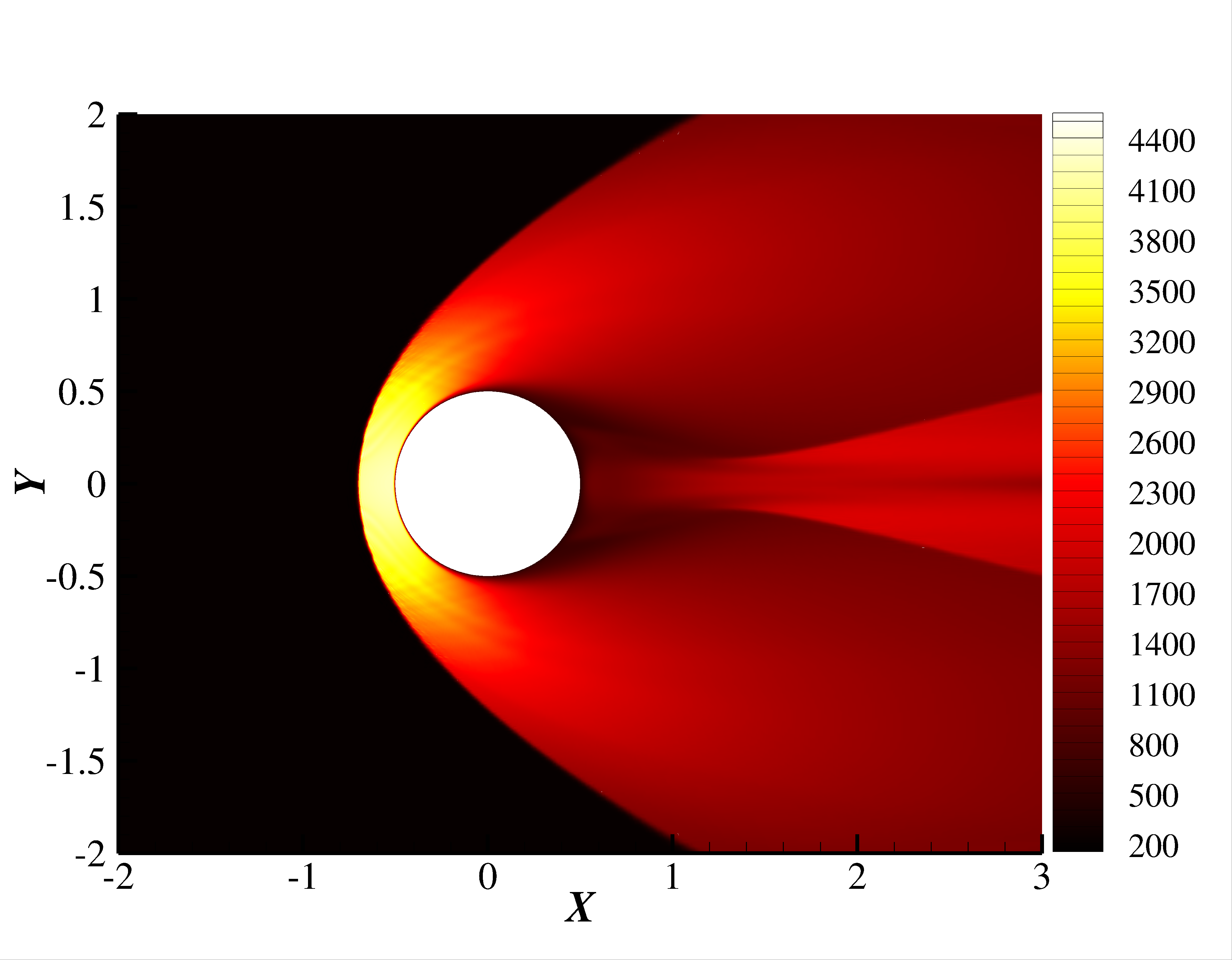}
  \end{minipage}
  \hfill
  \begin{minipage}{0.49\textwidth}
    \includegraphics[width=\textwidth]{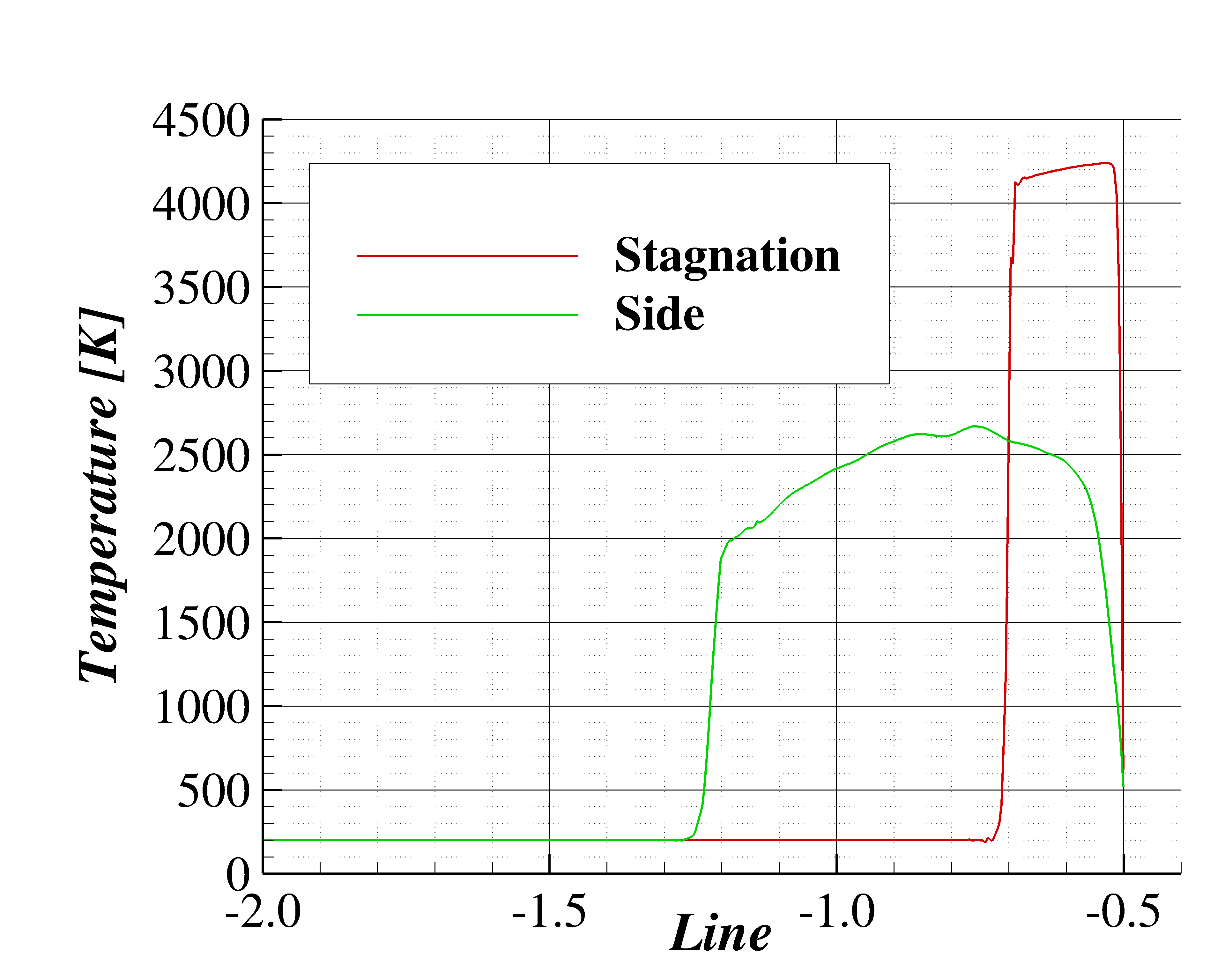}
  \end{minipage}
  \caption{Temperature field, stagnation ($x=[-2,-0.5]$, $y=0$), and the side line ($x=0$,$y=[-2,-0.5]$) temperatures at $t=10$ for the supersonic cylinder at $Ma=10$ with $p=4$.}
  \label{fig:cylinder_Ma10_flow_lines}
\end{figure}

The case is advanced up to time $t=10$, with the residual reaching approximately $3\times10^{-1}$. The resulting flow at $t=10$ is shown in the left panel of Figure \ref{fig:cylinder_Ma10_flow_lines}, and the temperature along the stagnation and side lines is shown in the right panel of Figure \ref{fig:cylinder_Ma10_flow_lines}. The results show that the scheme can enforce the isothermal condition at the cylinder wall, with a temperature difference of $\sim 3700 K$ between the stagnation bubble and the wall. Additionally, the cooler region below the cylinder is shown, where the isothermal condition enforces the correct temperature at the wall. The results also show that the scheme remains stable even in the presence of extreme wall heat fluxes.


\subsection{Thermal Couette flow}

\begin{figure}
  \centering
  \begin{minipage}{0.49\textwidth}
    \includegraphics[width=\textwidth]{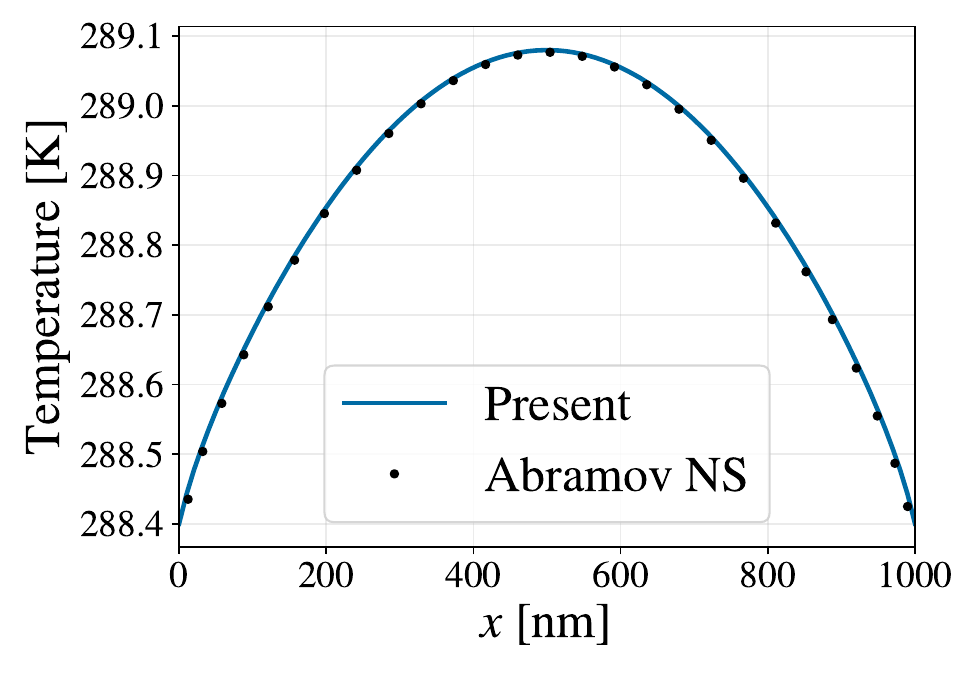}
  \end{minipage}
  \hfill
  \begin{minipage}{0.49\textwidth}
    \includegraphics[width=\textwidth]{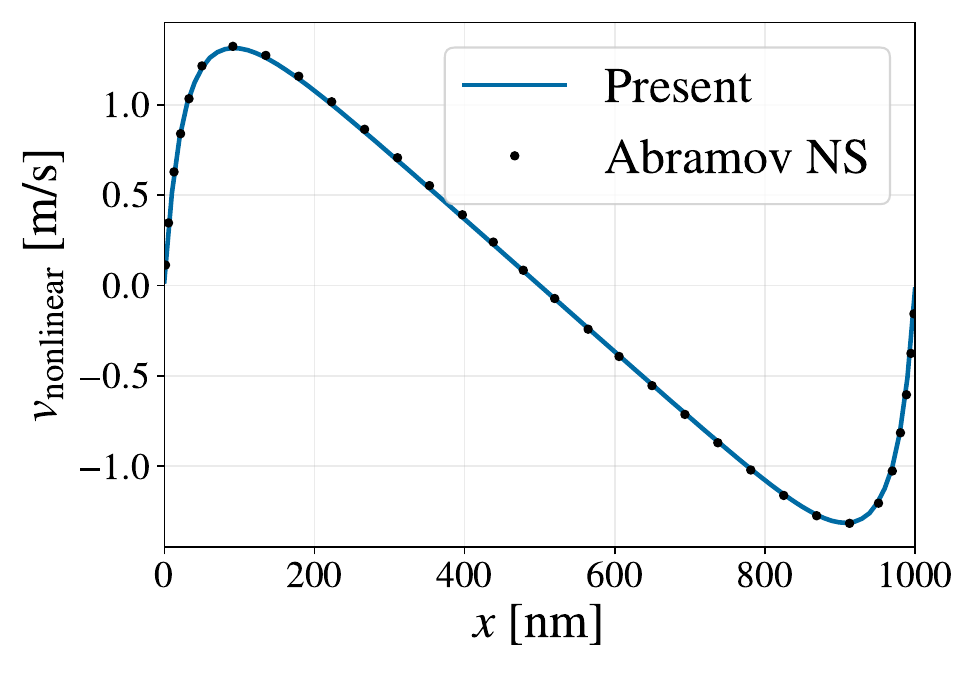}
  \end{minipage}
  \caption{Profiles of temperature (Kelvin) and nonlinear velocity component ($v_\text{nonlinear} = v - v_\text{linear}$) at steady state for the thermal Couette flow $N=32$ with $p=4$ and $\alpha=\Delta x^{-2}$.}
    \label{fig:couette_flow_lines}
\end{figure}

The thermal Couette flow is a canonical test case for validating the implementation of viscous and thermal boundary conditions. It is used to compare the accuracy against the Navier--Stokes result in \cite{abramov2017diffusive}. The ideal gas flow is charachterized by $Re=3.03$, $Pr=0.69$, and $Ma=0.131$, the reference temperature is $T_\text{ref}=288.15K$ and the wall temperature is $T^\text{wall}=288.4K$. The case is initialized with a velocity, $v_\text{linear}$, which is the slope of a line from $-45.45\frac{m}{s}$ to $45.45\frac{m}{s}$ such that the moving-wall velocities are satisfied at the initial state. The temperature is initialized as the reference temperature. Then, the case is advanced until steady state.

The viscosity in this case is defined by a volume viscosity and a near-wall scaled viscosity. The volume viscosity is defined as
\begin{equation}
    \mu = \sqrt{T}.
\end{equation}
A near-wall scaling is applied to the viscosity to account for the near-wall viscous effects, and is defined as
\begin{equation}
    \frac{\mu_\text{wall}}{\mu} = 1 + \frac12 \left(\frac{d_x}{\lambda} E_1(d_x/\lambda) - e^{-d_x/\lambda} \right),
\end{equation}
where $d_x$ is the distance from the wall, and
\begin{equation}
    E_1(x) = \int_{d_x}^\infty \frac{e^{-t}}{t} \, dt, 
    \quad\quad \lambda = \frac{Ma}{Re} \frac{\mu}{\rho} \sqrt{\frac{\gamma \pi}{2 T}}.
\end{equation}
The exponential integral $E_1(x)$ is approximated using a series expansion for small arguments and an asymptotic expansion for large arguments.

The domain consists of two parallel plates separated by a non-dimensional distance of $1$ ($1000nm$). The boundary conditions are an isothermal moving wall with $v_1=45.45\frac{m}{s}$ at the top plate and $v_1=-45.45\frac{m}{s}$ at the bottom plate, and periodic in the streamwise direction. The mesh is uniform, with $N=32$ elements in the wall-normal direction and $N=16$ in the streamwise direction, resulting in $4:1$ rectangular cells. Due to the stiffness of the problem, the isothermal coefficient is set to $\alpha=\Delta x^{-2}$. The case is advanced until non-dimensional time $t=1$ with residual $2\times10^{-5}$.

The case is advanced until $t=1$ using the implicit backward Euler scheme. Figures \ref{fig:couette_flow_lines} show the temperature and nonlinear velocity profiles. The linear velocity dominates the velocity magnitude, thus it is subtracted to reveal the nonlinear part that represents the viscous clinging to the moving walls. The results show that the scheme can enforce the isothermal condition at the walls without altering the expected velocity profile. The solution obtained is identical to the Navier--Stokes result in \cite{abramov2017diffusive}.


\section{Conclusions}
\label{sec:conclusions}

A novel entropy-stable high-order isothermal penalty is derived, and its stability and accuracy results are demonstrated. The penalty is derived as a temperature correction to the Fourier term, resulting in an isothermal consistent temperature gradient. A ballistic energy estimate is used to bound the $L^2$ norm of the solution, both continuously and discretely via SBP-SAT operators. The results show that this penalty is design-order accurate, even in the presence of a square wall. In addition, the convergence of the bound on the ballistic energy shows that the scheme is structure-preserving.

While the design-order property requires $\alpha=\Delta x^{-p}$, resulting in very stiff problems that require implicit time-steppers, this remains a limitation of this scheme for a case such as the $Ma=10$ flow around a heated cylinder. However, enforcing the isothermal condition can still be achieved with a relaxed parameter $\alpha=\Delta x_\text{avg}^{-q}$, where $\Delta x_\text{avg}$ is the average spacing in the domain and $q<p$. The latter approach is sufficient to enforce the isothermal condition, as in the case of the thermal Couette flow, which was shown to match the reference solution. Even under extreme flow conditions such as the $Ma=10$ flow around a heated cylinder, where the temperature difference at the stagnation point is $\sim3700 K$. 

Enforcing this penalty with design-order is observed to be stiff for all problems; subsonic and low-Mach cases that do not require positivity-preserving mechanisms may use a fully implicit scheme, as in the cases of subsonic pipe flow, subsonic square cylinder, and thermal Couette flow. However, higher-Mach cases, such as the $Ma=10$ bow shock flow, must use positivity-preserving schemes. The dual time-stepping positivity-preserving schemes~\cite{sayyari2026implicit,sayyari2026bdf2} can be used in this case to alleviate stiffness.




\section*{Acknowledgments}
The second author gratefully acknowledges the support from Department of Defense through grant W911NF2310183. 


\bibliographystyle{siamplain}
\bibliography{references}


\end{document}

%% file: positivity3D_def.tex
\usepackage{bm}
\usepackage{xspace}

\newcommand{\mr}[1]{\ensuremath{\mathrm{#1}}}
\newcommand{\mat}[1]{\ensuremath{\mathsf{#1}}}
\newcommand{\fnc}[1]{\ensuremath{\mathit{#1}}}
\newcommand{\bfnc}[1]{\ensuremath{\bm{\mathit{#1}}}}

\newcommand{\U}[0]{\ensuremath{\bfnc{U}}}

\newcommand{\Efnc}[0]{\ensuremath{\fnc{E}}}



%% file: references.bib
@phdthesis{upperman2021high,
  title={High-Order Positivity-Preserving L 2-Stable Spectral Collocation Schemes for the 3-D Compressible Navier-Stokes Equations},
  author={Upperman, Johnathon Keith},
  year={2021},
	school={Old Dominion University},
}

@article{abramov2017diffusive,
  title={Diffusive Boltzmann equation, its fluid dynamics, Couette flow and Knudsen layers},
  author={Abramov, Rafail V},
  journal={Physica A: Statistical Mechanics and its Applications},
  volume={484},
  pages={532--557},
  year={2017},
	publisher={Elsevier},
	doi = {10.1016/j.physa.2017.05.008},
	url = {https://doi.org/10.1016/j.physa.2017.05.008},
}

@article{brenner2005navier,
  title={Navier--stokes revisited},
  author={Brenner, Howard},
  journal={Physica A: Statistical Mechanics and its Applications},
  volume={349},
  number={1-2},
  pages={60--132},
  year={2005},
	publisher={Elsevier},
	doi = {10.1016/j.physa.2004.10.034},
	url = {https://doi.org/10.1016/j.physa.2004.10.034},
}

@misc{peng2024extensions,
  title={Extensions combining Chapman--Enskog and Brenner closure},
  author={Peng, N. and Han, G. and Ma, Z. and Chen, F.},
  year={2024},
  eprint={2412.18364},
  archivePrefix={arXiv},
  primaryClass={physics.flu-dyn},
  doi={10.48550/arXiv.2412.18364},
  url={https://doi.org/10.48550/arXiv.2412.18364}
}

@article{dalcin2019conservative,
  title={Conservative and entropy stable solid wall boundary conditions for the compressible {N}avier--{S}tokes equations: Adiabatic wall and heat entropy transfer},
  author={Dalcin, Lisandro and Rojas, Diego and Zampini, Stefano and Fern{\'a}ndez, David C Del Rey and Carpenter, Mark H and Parsani, Matteo},
  journal={Journal of Computational Physics},
  volume={397},
  pages={108775},
  year={2019},
  publisher={Elsevier}
}

@techreport{parsani2014entropy,
  title={Entropy stable wall boundary conditions for the compressible {N}avier--{S}tokes equations},
  author={Parsani, Matteo and Carpenter, Mark H and Nielsen, Eric J},
  year={2014}
}

@article{chan2022entropy,
  title={Entropy stable modal discontinuous {G}alerkin schemes and wall boundary conditions for the compressible {N}avier--{S}tokes equations},
  author={Chan, Jesse and Lin, Yimin and Warburton, Tim},
  journal={Journal of Computational Physics},
  volume={448},
  pages={110723},
  year={2022},
  publisher={Elsevier}
}

@article{fernandez2014review,
  author =        {Fern{\'a}ndez, David C Del Rey and Hicken, Jason E and
                   Zingg, David W},
  journal =       {Computers \& Fluids},
  pages =         {171--196},
  publisher =     {Elsevier},
  title =         {Review of summation-by-parts operators with
                   simultaneous approximation terms for the numerical
                   solution of partial differential equations},
  volume =        {95},
  year =          {2014},
}

@article{upperman2023first,
  author =        {Upperman, Johnathon and Yamaleev, Nail K},
  journal =       {Journal of Scientific Computing},
  number =        {1},
  pages =         {18},
  publisher =     {Springer},
  title =         {First-Order Positivity-Preserving Entropy Stable
                   Scheme for the 3-{D} Compressible {N}avier--{S}tokes
                   Equations},
  volume =        {94},
  year =          {2023},
}

@article{yamaleev2023high,
  author =        {Yamaleev, Nail K and Upperman, Johnathon},
  journal =       {Journal of Scientific Computing},
  number =        {1},
  pages =         {11},
  publisher =     {Springer},
  title =         {High-order positivity-preserving entropy stable
                   schemes for the 3-{D} compressible {N}avier--{S}tokes
                   equations},
  volume =        {95},
  year =          {2023},
}

@article{feireisl2010new,
  author =        {Feireisl, Eduard and Vasseur, Alexis},
  journal =       {New Directions in Mathematical Fluid Mechanics: The
                   Alexander V. Kazhikhov Memorial Volume},
  pages =         {153--179},
  publisher =     {Springer},
  title =         {New perspectives in fluid dynamics: Mathematical
                   analysis of a model proposed by {H}oward {B}renner},
  year =          {2010},
}

@article{carpenter2014entropy,
  author =        {Carpenter, Mark H and Fisher, Travis C and
                   Nielsen, Eric J and Frankel, Steven H},
  journal =       {SIAM Journal on Scientific Computing},
  number =        {5},
  pages =         {B835--B867},
  publisher =     {SIAM},
  title =         {Entropy stable spectral collocation schemes for the
                   {N}avier--{S}tokes equations: Discontinuous
                   interfaces},
  volume =        {36},
  year =          {2014},
}

@article{fisher2013discretely,
  author =        {Fisher, Travis C and Carpenter, Mark H and
                   Nordstr{\"o}m, Jan and Yamaleev, Nail K and
                   Swanson, Charles},
  journal =       {Journal of Computational Physics},
  pages =         {353--375},
  publisher =     {Elsevier},
  title =         {Discretely conservative finite-difference
                   formulations for nonlinear conservation laws in split
                   form: Theory and boundary conditions},
  volume =        {234},
  year =          {2013},
}

@article{fisher2011boundary,
  author =        {Fisher, Travis C and Carpenter, Mark H and
                   Yamaleev, Nail K and Frankel, Steven H},
  journal =       {Journal of Computational Physics},
  number =        {10},
  pages =         {3727--3752},
  publisher =     {Elsevier},
  title =         {Boundary closures for fourth-order energy stable
                   weighted essentially non-oscillatory
                   finite-difference schemes},
  volume =        {230},
  year =          {2011},
}

@inproceedings{carpenter2016entropy,
  author =        {M. H. Carpenter and T. Fisher and E. Nielsen and
                   M. Parsani and M. Svard and N. Yamaleev},
  booktitle =     {Handbook on Numerical Analysis},
  editor =        {R. Abgrall and C.-W. Shu},
  pages =         {495--524},
  publisher =     {Elsevier},
  title =         {Entropy stable summation-by-parts formulations for
                   compressible computational fluid dynamics},
  year =          {2016},
}

@article{merriam1989entropy,
  author =        {M. {L}. Merriam},
  journal =       {NASA-TM--1989--101086},
  title =         {An Entropy-Based Approach to Nonlinear Stability},
  year =          {1989},
}

@book{schlichting2016boundary,
  title={Boundary-layer theory},
  author={Schlichting, Hermann and Gersten, Klaus},
  year={2016},
  publisher={springer}
}

@article{lighthill1950contributions,
  title={Contributions to the theory of heat transfer through a laminar boundary layer},
  author={Lighthill, Michael James},
  journal={Proceedings of the Royal Society of London. Series A. Mathematical and Physical Sciences},
  volume={202},
  number={1070},
  pages={359--377},
  year={1950},
  publisher={The Royal Society London}
}

@incollection{ostrach1972natural,
  title={Natural convection in enclosures},
  author={Ostrach, Simon},
  booktitle={Advances in heat transfer},
  volume={8},
  pages={161--227},
  year={1972},
  publisher={Elsevier}
}

@article{paolucci1994differentially,
  title={The differentially heated cavity},
  author={Paolucci, S},
  journal={Sadhana},
  volume={19},
  number={5},
  pages={619--647},
  year={1994},
  publisher={Springer}
}

@article{spina1994physics,
  title={The physics of supersonic turbulent boundary layers},
  author={Spina, Eric F and Smits, Alexander J and Robinson, Stephen K},
  journal={Annual review of fluid mechanics},
  volume={26},
  number={1},
  pages={287--319},
  year={1994},
  publisher={Annual Reviews 4139 El Camino Way, PO Box 10139, Palo Alto, CA 94303-0139, USA}
}

@book{anderson2006hypersonic,
  title={Hypersonic and High-Temperature Gas Dynamics},
  author={Anderson Jr, John D},
  year={2006},
  publisher={American Institute of Aeronautics and Astronautics}
}

@article{pirozzoli2011numerical,
  title={Numerical methods for high-speed flows},
  author={Pirozzoli, Sergio},
  journal={Annual review of fluid mechanics},
  volume={43},
  number={1},
  pages={163--194},
  year={2011},
  publisher={Annual Reviews}
}

@article{wang2013high,
  title={High-order CFD methods: current status and perspective},
  author={Wang, Zhijian J and Fidkowski, Krzysztof and Abgrall, R{\'e}mi and Bassi, Francesco and Caraeni, Doru and Cary, Andrew and Deconinck, Herman and Hartmann, Ralf and Hillewaert, Koen and Huynh, Hung T and others},
  journal={International Journal for Numerical Methods in Fluids},
  volume={72},
  number={8},
  pages={811--845},
  year={2013},
  publisher={Wiley Online Library}
}

@inproceedings{persson2006sub,
  title={Sub-cell shock capturing for discontinuous Galerkin methods},
  author={Persson, Per-Olof and Peraire, Jaime},
  booktitle={44th AIAA aerospace sciences meeting and exhibit},
  pages={112},
  year={2006}
}

@article{chaudhuri2022navier,
  title={Navier--Stokes--Fourier system with Dirichlet boundary conditions},
  author={Chaudhuri, Nilasis and Feireisl, Eduard},
  journal={Applicable Analysis},
  volume={101},
  number={12},
  pages={4076--4094},
  year={2022},
  publisher={Taylor \& Francis}
}

@article{sayyari2022capabilities,
  title={The capabilities of summation-by-parts and structure-preserving operators for compressible computational fluid dynamics and reaction-diffusion models},
  author={Sayyari, Mohammed},
  year={2022}
}

@article{svard2015weak,
  title={Weak solutions and convergent numerical schemes of modified compressible Navier--Stokes equations},
  author={Sv{\"a}rd, Magnus},
  journal={Journal of Computational Physics},
  volume={288},
  pages={19--51},
  year={2015},
  publisher={Elsevier}
}

@article{geuzaine2009gmsh,
  title={Gmsh: A 3-{D} finite element mesh generator with built-in pre-and post-processing facilities},
  author={Geuzaine, Christophe and Remacle, Jean-Fran{\c{c}}ois},
  journal={International journal for numerical methods in engineering},
  volume={79},
  number={11},
  pages={1309--1331},
  year={2009},
  publisher={Wiley Online Library}
}

@article{petsc-ts,
	author={Abhyankar, Shrirang and Brown, Jed and Constantinescu, Emil M and Ghosh, Debojyoti and Smith, Barry F and Zhang, Hong},
	title={PETSc/TS: A Modern Scalable ODE/DAE Solver Library},
	journal={arXiv preprint arXiv:1806.01437},
	year={2018}
}

@article{hunter2007matplotlib,
  title={Matplotlib: A 2D graphics environment},
  author={Hunter, John D},
  journal={Computing in science \& engineering},
  volume={9},
  number={03},
  pages={90--95},
  year={2007},
  publisher={IEEE Computer Society}
}

@article{sayyari2026implicit,
  title={Implicit dual time-stepping positivity-preserving entropy-stable schemes for the compressible Navier-Stokes equations},
  author={Sayyari, Mohammed and Yamaleev, Nail K},
  journal={Journal of Computational and Applied Mathematics},
  pages={117735},
  year={2026},
  publisher={Elsevier}
}

@article{sayyari2026bdf2,
  title={Implicit BDF2 dual time-stepping positivity-preserving entropy-stable schemes for unsteady compressible viscous flows},
  author={Sayyari, Mohammed and Yamaleev, Nail K},
  journal={arXiv preprint arXiv:2608.20103},
  year={2026}
}
